\documentclass{article}
\usepackage{amsmath,amssymb,latexsym,amsthm}
\usepackage{mathrsfs}
\usepackage{mathtools}
\usepackage[english]{babel}
\usepackage[latin2]{inputenc}
\usepackage[margin=4cm]{geometry} 

\usepackage{graphicx} 
\usepackage{amssymb}
\usepackage{amsmath}
\usepackage{amsthm}
\usepackage{cancel}
\usepackage{xcolor}
\usepackage{pdflscape}
\numberwithin{equation}{subsubsection}

\theoremstyle{definition}
\newtheorem{theor}{Theorem}
\newtheorem{lem}[theor]{Lemma}
\newtheorem{co}[theor]{Corollary}
\newtheorem{prop}[theor]{Proposition}

\newtheorem{remark}[theor]{Remark}

\title{Three-dimensional simple real Bol algebras}
\author{G. Falcone, \'A. Figula, E. K\'asa, G. Mattana, P. T. Nagy}

\begin{document}

\maketitle

\begin{abstract}
In this paper, we provide a complete classification of real three-dimensional simple Bol algebras $(B, [.,.], \langle .,.,. \rangle )$. The main part of the classification concerns the case $[B,B]=B$. In addition, we identify two isomorphism classes with a nonzero binary product, both satisfying $\dim[B,B]=2$, as well as four simple classes of Lie triple systems in which the binary product vanishes.
\end{abstract}



\section{Introduction}
\subsection{Background}
Non-associative algebraic structures arise naturally in differential geometry and mathematical physics whenever the local composition of transformations fails to be associative. Among these, \emph{Bol algebras} provide the infinitesimal counterpart of smooth Bol loops (named so after G. Bol \cite{Bol}), extending the classical correspondence between Lie groups and Lie algebras to a non-associative setting. This framework, originating in the work of Mal'tsev \cite{Malcev} and developed systematically by Akivis \cite{Akivis, Akivis1, Akivis2} and Sabinin \cite{sabinin,Sabinin1}, reveals a deep connection between algebraic identities and the local differential geometry of quasigroups.

At the global level, the theory of simple Bol loops has also attracted considerable attention; in particular, G. Nagy constructed simple Bol loops, including finite non-Moufang examples \cite{GNagy, GNagyTAMS}. 
Interest in Bol algebras has continued during the years, as witnessed by \cite{Filippov, Hentzel, Mikheev, Izquierdo, Zhang}.

From a geometric viewpoint, Bol algebras encode the infinitesimal structure of local homogeneous spaces endowed with a non-associative multiplication \cite{Nagy}. 
In particular, Lie triple systems, which can be regarded as Bol algebras with vanishing binary product \([x,y]=0\), are related to symmetric spaces \cite{Cartan, Jacobson, Loos, Nagy}. 
The Bol identities reflect the differential geometry of the corresponding local loops. In the geodesic-loop approach, the binary and trilinear infinitesimal operations are expressed in terms of the torsion and curvature of an affine connection and their covariant derivatives \cite[Corollary 5.14]{Nagy}, \cite{sabinin}. In this setting, the binary and trilinear operations of a Bol algebra encode the geometric properties of the underlying connection 
and provide a natural algebraic framework for describing affine connections with torsion \cite{Akivis0}, \cite[Chapter 4, $\S5$]{Gorbatsevich}, \cite{sabinin}, which play an important role in geometric mechanics.


\subsection{Overview}

The $2$-dimensional real Bol algebras are determined in \cite{kuzmin}, the nilpotent Bol algebras up to dimension $4$ are classified in \cite{abdelwahab,abdelwahab0,Abdurasulov}, and the determination of the $3$-dimensional solvable and splitting real Lie triple systems is carried out in \cite{Bouetou1}. 
The purpose of this paper is to provide a complete classification of
real three-dimensional simple Bol algebras.
This was made possible by a particularly favorable circumstance: despite the apparent complexity of the Bol identities (see Section \ref{sec:2}, (B1), (B2), (C)), the Hodge identification of $\Lambda^2B$ with $B$, which is available only in dimension three, makes a complete classification unexpectedly tractable.

The case where 
$[B,B]=B$
constitutes the main part of the classification. It does not, however,
exhaust all simple $3$-dimensional Bol algebras (see Theorems \ref{thm:rank-two} and \ref{thm:simpleLTS}).  There are exactly two further isomorphism
classes with nonzero binary product, both satisfying $
\dim[B,B]=2$,
whereas, notably, no simple Bol algebra occurs when
$\dim[B,B]=1$: 
one-dimensional binary product is incompatible with simplicity in dimension three. 
Finally, when
$[B,B]=0$, 
the binary product vanishes, and one is precisely in the class of
Lie triple systems; this yields four further simple cases  (see also \cite{Lister}).


\medskip

Our methods can be summarized as follows. Fixing an oriented Euclidean
structure on the underlying vector space $B$, in the main case $[B,B]=B$, the bilinear product can be written as
\[
[x,y]=A(x\times y), 
\]
where $A \in \operatorname{GL}(3,\mathbb R)$. Under a change of basis $g\in\operatorname{GL}(3,\mathbb R)$, the
matrix $A$ transforms according to the twisted congruence action
\[
A\longmapsto(\det g)^{-1}gAg^T
\]
(see \eqref{eq:action}). For nonsingular $A$, this equivalence is shown to coincide with ordinary
real congruence up to multiplication by $-1$ (see Lemma \ref{lem:twisted-congruence}). The canonical forms of
Lee--Weinberg therefore reduce the classification of the binary product
to a finite collection of normal forms, some of which depend on a real
parameter (see Theorem \ref{thm:binary-normal-forms}).

The trilinear product is encoded by a linear map
\[
P:B\longrightarrow\operatorname{End}(B),
\qquad
\langle x,y,z \rangle =P(x\times y)z,
\]
and hence by the three matrices $
P_i=P(e_i)$, $i=1,2,3$, where $\{e_1, e_2, e_3\}$ is a basis.
The Bol identities can consequently be translated into explicit
algebraic conditions on the quadruple
$(A;P_1,P_2,P_3)$.

A central feature of the classification is that the first
Bol identity (B1), together with cyclicity (C), already provides the decisive
restrictions. For each canonical matrix $A$, these identities form a
linear system in the entries of $P_1,P_2,P_3$. In the two cases where $A$ is symmetric, 
this system is rigid: the trilinear product is uniquely determined and
is necessarily of Lie type (see Proposition \ref{prop:symmetric-cases}),
\[
\langle x,y,z \rangle =[[x,y],z].
\]
In the nonsymmetric cases the same analysis either produces explicit
families of compatible trilinear products or rules out the corresponding
binary normal form altogether. In particular, the cases
{\rm (II.3$^\pm$)}, {\rm (II.4)}, and {\rm (III.1)} are already
excluded by the first Bol identity and cyclicity. Although these
constraints can be expressed as a $27\times27$ linear system, we avoid
a brute-force elimination and instead derive the solutions directly
from the operator identities. The resulting proofs are conceptual and
readable, 
while the corresponding linear systems have been independently checked both by hand and by symbolic computation.

A useful simplification occurs in this analysis: for every
family surviving the first Bol identity and cyclicity, the second Bol
identity (B2) is satisfied identically and introduces no further
restriction.   This phenomenon no longer persists
when $\dim[B,B] <3$, where the second Bol identity plays an
essential role. 

The preceding analysis also allows us to clarify the relation between
simplicity and the rank of the binary product. Since
\[
[B,B]=\operatorname{im}A,
\]
the condition $\operatorname{rank}A=3$ is equivalent to $[B,B]=B$.
In dimension three we prove that
$
[B,B]=B$ implies that $B$ is simple (see Proposition \ref{simpleBBB}).
The converse has some exceptions: simple Bol algebras also occur when
$\operatorname{rank}A=2$ and when $\operatorname{rank}A=0$.
The rank-two analysis yields exactly two further isomorphism classes,
while rank one yields no simple algebra.
Finally, if $A=0$, the binary product vanishes and the Bol identities
reduce to those of a Lie triple system. Thus the rank-zero case is
precisely the problem of classifying simple three-dimensional real Lie
triple systems. Using Lister's structure theory, these are represented
by
\[
\langle x,y,z\rangle
=h(y,z)x-h(x,z)y,
\]
where, up to congruence, 
$h=\pm I_3$, or 
$h=\pm \operatorname{diag}(1,1,-1)$.

\medskip

The explicit parametrization of all solutions also
opens the way to further investigations, including the study of
integrability to global Bol loops (cf. e.g. \cite{figula1}), the analysis of invariant
connections, and the relation with symmetric and pseudo-Riemannian
geometries.

\subsection{Main result}

Because of the technical nature and unavoidable length of the following
analysis, for ease of reference we collect here the main conclusions.

\bigskip\noindent
\textbf{Classification of simple three-dimensional real Bol algebras.}

Let $B$ be a simple three-dimensional real Bol algebra with binary
product $[x,y]$ and ternary product $\langle x,y,z\rangle$.  Then, up
to isomorphism, exactly one of the following cases occurs.

\medskip\noindent
\medskip\noindent
(i)
$[B,B]=B$. In this case, in a suitable basis, the bilinear and trilinear products are
precisely those listed in the following table:

\medskip\noindent
$\renewcommand{\arraystretch}{1.4}
\begin{array}{|c|c|c|}
\hline
\textbf{\!\!Case\!\!}
&
{[x,y]=A(x\times y)}
&
\begin{array}{c}
{\langle x,y,z\rangle=P(x\times y)z},\\
\mbox{where $P$ is given by }
\{P_1,P_2,P_3\},\ \mbox{with }P_i:=P(e_i)
\end{array}
\\
\hline

\mathrm{(I.1)}
&
A=I_3
&
[[x,y],z]=(x\times y)\times z,
\quad\mbox{i.e. }\mathfrak{so}(3)
\\
\hline

\mathrm{(I.2)}
&
A=\operatorname{diag}(1,1,-1)
&
[[x,y],z],
\quad\mbox{i.e. }\mathfrak{so}(2,1)\simeq\mathfrak{sl}_2(\mathbb R)
\\
\hline

\mathrm{(II.1)}
&
{\text{\scriptsize
$A=
\begin{pmatrix}
0&-1&0\\
1&0&0\\
0&0&1
\end{pmatrix}$}}
&
\begin{array}{c}
\\[-4mm]
{\text{\scriptsize
$\left\{
\begin{pmatrix}
0&0&0\\
0&0&1\\
0&-\alpha&0
\end{pmatrix},
\;
\begin{pmatrix}
0&0&-1\\
0&0&0\\
-\beta&0&0
\end{pmatrix},
\;
\begin{pmatrix}
0&\alpha&0\\
\beta&0&0\\
0&0&0
\end{pmatrix}
\right\}$}}
\\[1mm]
\mbox{with }
(\alpha,\beta)\in
\{(\pm1,0),(a,a),(b,-b)\},
\quad a>0,\ b\in\mathbb R
\end{array}
\\
\hline

\mathrm{(II.2^{\pm})}
&
{\text{\scriptsize
$A=
\begin{pmatrix}
\pm1&-1&0\\
1&0&0\\
0&0&1
\end{pmatrix}$}}
&
\begin{array}{c}
\\[-4mm]
{\text{\scriptsize
$\left\{
\begin{pmatrix}
0&0&0\\
0&0&1\\
1&0&0
\end{pmatrix},
\quad
\begin{pmatrix}
0&0&-1\\
0&0&0\\
0&-1&0
\end{pmatrix},
\quad
\begin{pmatrix}
-1&0&0\\
0&1&0\\
0&0&0
\end{pmatrix}
\right\}$}}
\\[5mm]
\end{array}
\\
\hline
\end{array}$

\bigskip\noindent
{\rm (ii)}
$\dim[B,B]=2$.
Then, in a suitable basis,
\[
A=
\begin{pmatrix}
0&1&0\\
0&0&0\\
0&0&1
\end{pmatrix},
\]
and the ternary product is represented by
\[
P_1=0,
\qquad
P_2=
\begin{pmatrix}
-\dfrac{2t}{3}&0&0\\[3mm]
\dfrac{t^3}{36}&\dfrac{t}{3}&0\\[3mm]
-\dfrac{5t^2}{12}&1&\dfrac{2t}{3}
\end{pmatrix},
\qquad
P_3=
\begin{pmatrix}
2&0&0\\[1mm]
0&-2&-\dfrac{t}{3}\\[1mm]
t&0&-1
\end{pmatrix},
\]
where
$t\in\{0,1\}$.
The two values $t=0$ and $t=1$ give non-isomorphic Bol algebras.

\bigskip\noindent
{\rm (iii)}
$[B,B]=0$.
Then the binary product vanishes identically,
$A=0$,
and $B$ is a simple three-dimensional real Lie triple system.  In a
suitable basis its ternary product is
\[
\langle x,y,z\rangle
=
h(y,z)x-h(x,z)y,
\]
where $h(u,v)=uhv^T$, and
$
h\in
\left\{
 I_3,\,
 -I_3,\,
 \operatorname{diag}(1,1,-1),\,
 \operatorname{diag}(-1,-1,1)
\right\}.$
These four choices give four pairwise non-isomorphic Lie triple
systems.


\smallskip\noindent
Conversely, every Bol algebra occurring in {\rm (i)}, {\rm (ii)}, or
{\rm (iii)} is simple.

\medskip
The proofs are in Sections \ref{sec2.2}, \ref{sec:3.2}, \ref{sec:3.3}, \ref{sec:3.4}.

\section{Three dimensional real Bol algebras with $[B,B]=B$} \label{sec:2}

Let \(B\) be a real Bol algebra, with a (skew-symmetric) bilinear product
\[
[\cdot,\cdot]:\Lambda^2 B\to B,
\]
and a trilinear product
\[
\langle \cdot,\cdot,\cdot\rangle: \Lambda^2 B\times B\to B,
\]
(skew-symmetric in the first two variables)   satisfying the cyclic identity 
\[\langle x,y,z \rangle + \langle y,z,x \rangle + \langle z,x,y \rangle =0\tag{{C}}\] 
and the Bol identities
\begin{align}
{\langle x,y,[z,t]\rangle}&
{=[\langle x,y,z\rangle,t]-[\langle x,y,t\rangle,z]+\langle z,t,[x,y]\rangle+\big[[x,y],[z,t]\big]}
\tag{{B1}}
\\
\langle x,y,\langle z,t,u\rangle\rangle
&=
\langle \langle x,y,z\rangle,t,u\rangle
+\langle z,\langle x,y,t\rangle,u\rangle
+\langle z,t,\langle x,y,u\rangle\rangle.
\tag{B2}
\end{align}


\subsection{The binary product}
We assume throughout the present Section \ref{sec:2} that $\dim B=\dim [B,B]=3$.

\medskip
Fix an inner product and an orientation on \(B\), or equivalently declare an orthonormal basis. Then the Hodge operator gives an isomorphism
\[
*:\Lambda^2B\to B.
\]
Equivalently, the associated vector product
\[
x\times y:=*(x\wedge y)
\]
is described with respect to the orthonormal basis by means of the usual formulas 
\[e_1\times e_2=e_3,\quad e_2\times e_3=e_1,\quad e_3\times e_1=e_2.\]
Since \([B,B]=B\), there exists a unique invertible linear map $A\in \operatorname{GL}(B)$
such that
\[
[x,y]=A(x\times y).
\]

\medskip
The first result we give is the following Proposition \ref{simpleBBB}, whose converse is true, but with some exceptions that will be pointed out in Section \ref{sec:simplicity}. 

\begin{prop} \label{simpleBBB}
Let $B$ be a three-dimensional Bol algebra. If
$
[B,B]=B$, 
then $B$ is simple.
\end{prop}

\begin{proof}
Suppose, by contradiction, that $B$ admits a nonzero proper Bol ideal $I$.
Then the quotient $B/I$ is a Bol algebra and its binary product is given by
\[
[x+I,y+I]=[x,y]+I.
\]
Consequently,
\[
[B/I,B/I]
=
\frac{[B,B]+I}{I}=\frac{B+I}{I}=\frac{B}{I}.
\]
On the other hand, since $B$ is three-dimensional and $I$ is nonzero and
proper, one has
\[
1\leq \dim(B/I)\leq 2.
\]
If $\operatorname{dim}(B/I)=1$, then $\operatorname{dim}[B/I,B/I]=0$. Also, if $\operatorname{dim}(B/I)=2$ and $B/I=\operatorname{Span}\{e_1,e_2\}$, then  $[B/I,B/I]=\operatorname{Span}\{[e_1,e_2]\}$ has dimension at most one.
Therefore
$[B/I,B/I]\neq B/I$
for every nonzero algebra $B/I$ of dimension at most two. Hence $B$ has no nonzero
proper ideals and is simple.
\end{proof}

Notice that \(A\) is not the tensor defining the binary product itself, but
its matrix after identifying \(\Lambda^2B\) with \(B\) by means of the Hodge
operator. Consequently, the natural action of $\operatorname{GL}(B)$ on the binary product
is not ordinary congruence 
between two matrices $A$ and $B$, that is, $B=kAk^T$ 
for some $k\in\operatorname{GL}_3(\mathbb R)$. 

Indeed, let $g\in \operatorname{GL}(B)$ be a change of basis. The binary product
transforms according to  
\[
[x,y]'
=
g[g^{-1}x,g^{-1}y] 
\]
(see \cite{Burde}). Moreover, the vector product satisfies
\[
(g^{-1}x)\times(g^{-1}y)
=
(\det g)^{-1}g^T(x\times y),
\]
which is equivalent to the classical identity
\[
\Lambda^2g^{-1}
=
(\det g)^{-1}g^T
\]
under the identification
\(\Lambda^2B\simeq B\).
Therefore
\[
[x,y]'
=
gA\!\left((\det g)^{-1}g^T(x\times y)\right),
\]
so that the matrix \(A\) transforms as
\begin{equation}\label{eq:action}
A\longmapsto
(\det g)^{-1}gAg^T. 
\end{equation}

However, the action \eqref{eq:action} associated with the binary product is not far from ordinary congruence, as the following shows:

\begin{lem}\label{lem:twisted-congruence}
If $A,B\in\operatorname{GL}_3(\mathbb R)$, then
\[
B=(\det g)^{-1}gAg^T
\]
for some $g\in\operatorname{GL}_3(\mathbb R)$, if and only if $B$ is congruent either
to $A$ or to $-A$.
\end{lem}

\begin{proof}
Suppose first that
\[
B=(\det g)^{-1}gAg^T.
\]
If $\det g>0$, then
\[
B=kAk^T,
\]
with $k=(\det g)^{-1/2}g$, whereas, if $\det g<0$, then
\[
B=-kAk^T,
\]
with $k=(-\det g)^{-1/2}g$. Thus $B$ is congruent either to $A$ or to $-A$.

Conversely, suppose that
$B=kAk^T$ and
put
$g=(\det k)^{-1}k$.
Since the dimension is three,
$\det g=(\det k)^{-2}$,
and hence
\[
(\det g)^{-1}gAg^T
=
(\det k)^2(\det k)^{-2}kAk^T
=
B.
\]
Finally, taking $g=-I_3$ gives
$(\det g)^{-1}gAg^T=-A$.
\end{proof}

\subsubsection{Normal forms of the binary operation}
\label{subsec:binary-normal-form}

Since, after a change of basis $g\in\operatorname{GL}_3(\mathbb R)$, the matrix associated to the binary product transforms into $\pm kAk^T$, with $k={(|\operatorname{det} g|)^{-1/2}}g$, the classification of the binary operation can be reduced to the
classification of real bilinear forms under congruence. We shall use the
explicit canonical form of Lee and Weinberg \cite[Theorem~II]{Lee}, obtained
from Thompson's classification of pencils of symmetric and skew-symmetric
matrices (cf. \cite{Thompson}, see also
\cite{Horn}).
Their theorem states that every real
matrix is congruent to a unique matrix in quasidiagonal form whose blocks
belong to the eight types
$\{m'_3,\,
\infty'_4,\,
\infty'_5,\,
o'_3,\,
o'_4,\,
\alpha'_3,\,
\beta'_4,\,
\beta'_5\}.$

For nonsingular matrices of order three the list simplifies drastically.
The block $m'_3$ is singular, while $o'_4$ and $\beta'_5$ have order at
least four. Thus only
\[\{
o'_3,\,
\infty'_4,\,
\infty'_5,\,
\alpha'_3,\,
\beta'_4\}
\]
can occur. More precisely, only the one- and three-dimensional
specializations of $o'_3$, and the two-dimensional specializations of
the other four types, are relevant.

\begin{theor}\label{thm:binary-normal-forms}
Consider the action of $\operatorname{GL}_3(\mathbb R)$ on $\operatorname{GL}_3(\mathbb R)$ defined by
\[
A\longmapsto(\det g)^{-1}gAg^T,
\qquad g\in\operatorname{GL}_3(\mathbb R).
\]
Then every invertible matrix is equivalent to exactly one matrix belonging
to one of the following nine types.

\medskip
\noindent
{\rm (I) Symmetric matrices.}

${\rm (I.1)}\qquad I_3$,

${\rm (I.2)}\qquad \operatorname{diag}(1,1,-1)$.

\medskip
\noindent
{\rm (II) Non-symmetric decomposable matrices.}

${\rm (II.1)}\qquad
\begin{pmatrix}
0&-1&0\\
1&0&0\\
0&0&1
\end{pmatrix}$,

${\rm (II.2}^{\pm}{\rm )}\qquad
\begin{pmatrix}
\pm1&-1&0\\
1&0&0\\
0&0&1
\end{pmatrix}$,

${\rm (II.3}^{\pm}{\rm )}\qquad
\begin{pmatrix}
\pm a&-1&0\\
1&\pm a&0\\
0&0&1
\end{pmatrix},
\qquad a>0$,

${\rm (II.4)}\qquad
\begin{pmatrix}
a&-1&0\\
1&-a&0\\
0&0&1
\end{pmatrix},
\qquad a>0,\quad a\ne1$.

\medskip
\noindent
{\rm (III) Indecomposable matrix.}

${\rm (III.1)}\qquad
\begin{pmatrix}
0&-1&1\\
1&1&0\\
1&0&0
\end{pmatrix}$.

\medskip\noindent
The two signs in {\rm (II.2)} and {\rm (II.3)} are non-equivalent, and
distinct values of $a$ in each one-parameter family give non-equivalent
matrices.
\end{theor}

\begin{proof}
By Lemma~\ref{lem:twisted-congruence}, it suffices to specialize the
Lee--Weinberg real-congruence canonical form to nonsingular matrices of
order three and then identify a matrix with its negative.

For later reference, we make the passage from the relevant
Lee--Weinberg blocks to our representatives explicit.

\medskip
\noindent
\emph{The case $1+1+1$.}
The one-dimensional specialization of $o'_3$ is
\(
[\pm 1],
\)
thus a decomposition into three one-dimensional blocks is
$\operatorname{diag}(\varepsilon_1,\varepsilon_2,\varepsilon_3)$, with
$\varepsilon_i=\pm1$.
Up to permutation and change of all signs, there are exactly
two possibilities, that is,
$I_3$, and $\operatorname{diag}(1,1,-1)$,
which give {\rm (I.1)} and {\rm (I.2)}.

\medskip
\noindent
\emph{The case $2+1$.}
Modulo the global sign, the one-dimensional summand may be normalized to
$[1]$. There are four possible nonsingular two-dimensional
Lee--Weinberg blocks.

\begin{itemize}
    \item 
For $e=1$, the block $\infty'_5$ is
\[
\begin{pmatrix}
0&1\\
-1&0
\end{pmatrix},\]
which is congruent to its opposite, because
\[
\begin{pmatrix}
0&1\\
-1&0
\end{pmatrix}=\begin{pmatrix}
1&0\\
0&-1
\end{pmatrix}\begin{pmatrix}
0&-1\\
1&0
\end{pmatrix}\begin{pmatrix}
1&0\\
0&-1
\end{pmatrix}^T.
\]
After adjoining the one-dimensional block $[1]$, we obtain precisely
{\rm (II.1)}.

\item For $e=2$, the block $\infty'_4$ is
\[
\varepsilon
\begin{pmatrix}
0&1\\
-1&1
\end{pmatrix},
\qquad
\varepsilon=\pm1.
\]
With 
$P_\varepsilon=
\begin{pmatrix}
0&1\\
\varepsilon&0
\end{pmatrix}$, one obtains
$P_\varepsilon
\left[
\varepsilon
\begin{pmatrix}
0&1\\
-1&1
\end{pmatrix}
\right]
P_\varepsilon^T
=
\begin{pmatrix}
\varepsilon&-1\\
1&0
\end{pmatrix}$, and after adjoining $[1]$ these are precisely
{\rm (II.2}$^\pm$\rm ).

\item For $e=1$, the block $\beta'_4$ reduces to
\[
B_{\varepsilon,b}
=
\varepsilon
\begin{pmatrix}
1&b\\
-b&1
\end{pmatrix},
\qquad
b>0,\qquad \varepsilon=\pm1.
\]
If
$a=b^{-1}$ and
$Q_\varepsilon=
\begin{pmatrix}
1&0\\
0&-\varepsilon
\end{pmatrix}$, then
$(\sqrt a\,Q_\varepsilon)
B_{\varepsilon,b}
(\sqrt a\,Q_\varepsilon)^T
=
\begin{pmatrix}
\varepsilon a&-1\\
1&\varepsilon a
\end{pmatrix}$.
Thus the two signed $\beta'_4$ families give exactly
{\rm (II.3}$^\pm$\rm ), with the bijective change of parameter
$a=b^{-1}$.

\item For $e=1$, the block $\alpha'_3$ is
\[
A_\alpha=
\begin{pmatrix}
0&\alpha+1\\
1-\alpha&0
\end{pmatrix},
\qquad \alpha\in\mathbb R\setminus\{0\}.
\]
Nonsingularity requires
$\alpha\ne\pm1$.
Moreover,
\[
P
(-A_\alpha)
P^T=A_{-\alpha},
\qquad
P=
\begin{pmatrix}
0&1\\
-1&0
\end{pmatrix}.
\]
Hence, after passing from congruence to our equivalence relation, we may
assume
$0<\alpha\ne1$.
Put
$a=\alpha^{-1}$
and
\[
P_\alpha=
\begin{pmatrix}
\dfrac1{2\alpha}&1\\[2mm]
\dfrac1{2\alpha}&-1
\end{pmatrix}.
\]
Then
$\det P_\alpha=-\alpha^{-1}\ne0$,
and direct multiplication gives
\[
P_\alpha A_\alpha P_\alpha^T
=
\begin{pmatrix}
a&-1\\
1&-a
\end{pmatrix}.
\]
Thus $\alpha'_3$ gives precisely {\rm (II.4)}, with
$0<\alpha^{-1}=a\ne1$.
\end{itemize}

\medskip
\noindent
\emph{The indecomposable $3$-dimensional case.}
The only nonsingular indecomposable Lee--Weinberg block of order three
is the $e=3$ specialization of $o'_3$. From their definition,
\[
o'_3
=
\varepsilon(\Delta_3+S\Delta_3)
=
\varepsilon
\begin{pmatrix}
0&0&1\\
0&1&1\\
1&-1&0
\end{pmatrix},
\qquad
\varepsilon=\pm1.
\]
If
$R=
\begin{pmatrix}
0&0&1\\
0&1&0\\
1&0&0
\end{pmatrix}$, then
$R
\begin{pmatrix}
0&0&1\\
0&1&1\\
1&-1&0
\end{pmatrix}
R^T
=
\begin{pmatrix}
0&-1&1\\
1&1&0\\
1&0&0
\end{pmatrix}$.

Thus the two signs of the Lee--Weinberg block give the two matrices
$\pm A_0$, where
\[
A_0=
\begin{pmatrix}
0&-1&1\\
1&1&0\\
1&0&0
\end{pmatrix}.
\]
They belong to the same orbit by
Lemma~\ref{lem:twisted-congruence}. Hence this gives the single case
{\rm (III.1)}.

\medskip
It remains to verify that passing
from ordinary congruence to congruence modulo sign introduces no further
identifications among the representatives displayed above.

The symmetric cases are distinguished by inertia. The cases
{\rm (II.1)}, {\rm (II.2)}, {\rm (II.3)}, and {\rm (II.4)} belong to
different Lee--Weinberg block types.

Within {\rm (II.2)}, the symmetric parts of the two $3\times3$ matrices
have inertia
$(2,0,1)$ and $(1,1,1)$,
respectively, where the third entry denotes the nullity. Multiplication by
$-1$ interchanges the positive and negative indices but preserves the
nullity, so these two classes remain distinct.

Within {\rm (II.3)}, the symmetric part of the positive-sign
representative is positive definite, whereas that of the negative-sign
representative has inertia
$(1,2)$.
After multiplication by $-1$ the latter has inertia $(2,1)$, and hence
the two signs again remain inequivalent.

Finally, the changes of parameter
$a=b^{-1}$
and $a=\alpha^{-1}$
are bijective on the chosen parameter ranges. Therefore distinct admissible
values of $a$ correspond to distinct Lee--Weinberg canonical parameters.
Moreover, multiplication by $-1$ preserves each of the four
Lee--Weinberg block families occurring in the $2+1$ decomposition
(up to the parameter identifications already made above), and hence
cannot produce identifications between distinct block types.

The uniqueness assertion in Lee--Weinberg's Theorem~II now shows that
the displayed list is complete and irredundant.
\end{proof}

\begin{remark}
We stress that, according to Theorem~\ref{thm:binary-normal-forms},
for every three-dimensional Bol algebra $B$ with $[B,B]=B$ there exists
a basis $\{e_1,e_2,e_3\}$ such that
\[
[e_1,e_2]=Ae_3,\qquad
[e_2,e_3]=Ae_1,\qquad
[e_3,e_1]=Ae_2,
\]
where $A$ is one of the normal forms listed in the theorem.
These are only the possible binary products \emph{a priori}: we shall
see that imposing identities \textup{(C)} and \textup{(B1)} rules out
the cases {\rm (II.3$^\pm$)}, {\rm (II.4)}, and {\rm (III.1)}.
\end{remark}

\subsection{Imposing the Bol identities}\label{sec2.2}

We now determine, for each of the canonical matrices listed in
Theorem~\ref{thm:binary-normal-forms}, all compatible ternary products.
Throughout this section we write
\[
 [x,y]=A(x\times y),\qquad
 \langle x,y,z\rangle=P(x\times y)z,
\]
where \(A\in \mathrm{GL}_3(\mathbb R)\) is one of the possible matrices of
Theorem~\ref{thm:binary-normal-forms} and
\[
 P\colon \mathbb R^3\longrightarrow \operatorname{M}_3(\mathbb R)
\]
is linear.  We put
$P_i:=P(e_i)$ for $i=1,2,3$.
Thus
$$\label{P_i}
 \langle e_2,e_3,z\rangle=P_1z,\qquad
 \langle e_3,e_1,z\rangle=P_2z,\qquad
 \langle e_1,e_2,z\rangle=P_3z.$$

The transformation law for \(P\) is obtained from
\[
 \langle x,y,z\rangle'
 =
 g\langle g^{-1}x,g^{-1}y,g^{-1}z\rangle
\]
and from
\[
 (g^{-1}x)\times(g^{-1}y)
 =
 (\det g)^{-1}g^T(x\times y).
\]
Accordingly,
\[
 P'(u)
 =
 g\,P\!\left((\det g)^{-1}g^Tu\right)g^{-1}.
\]
In particular, after the binary product has been reduced to a fixed canonical
matrix \(A\), the remaining equivalences are induced by the stabilizer
\[
 \operatorname{Stab}(A)
 =
 \left\{
 g\in\mathrm{GL}_3(\mathbb R):
 (\det g)^{-1}gAg^T=A
 \right\}.
\]

We next rewrite the Bol identities in terms of \(A\) and \(P\), converting an unpleasant coordinate classification into a manageable matrix problem.

\begin{remark}
For each fixed canonical matrix $A$, identity
\textup{(B1)} may of course be expanded in the canonical
basis into a linear system in the $27$ entries of
$(P_1,P_2,P_3)$. We have used this system, both by direct computation
and by symbolic calculation, as an independent check of the results
below. We deliberately do not use the resulting row reductions as
proofs: instead, the operator form of the Bol identities allows us to
derive the solution spaces directly and, in particular, to exhibit the
structural reasons for the rigidity or nonexistence occurring in the
individual cases.
\end{remark}

\begin{prop}\label{prop:operator-identities}
Since
$[x,y]=A(x\times y)$ and $\langle x,y,z\rangle=P(x\times y)z$,
the first Bol identity is equivalent to
\[
 \bigl(
 P(\alpha)A-\operatorname{tr}(P(\alpha))A
 +AP(\alpha)^T
 \bigr)\beta
 =
P(\beta)A\alpha+A(A\alpha\times A\beta)
 \tag{B1$_{\mathrm{op}}$}
\]
for all \(\alpha,\beta\in\mathbb R^3\).  The cyclic identity is equivalent to
\[
 P_1e_1+P_2e_2+P_3e_3=0.
 \tag{C$_{\mathrm{op}}$}
\]
Finally, the second Bol identity is equivalent to
\[
 [P(\alpha),P(\beta)]
 =
 P\!\left(
 \bigl(\operatorname{tr}(P(\alpha))I-P(\alpha)^T\bigr)\beta
 \right)
 \tag{B2$_{\mathrm{op}}$}
\]
for all \(\alpha,\beta\in\mathbb R^3\).
Equivalently,
\[
 [P_i,P_j]
 =
 \operatorname{tr}(P_i)P_j
 -
 \sum_{k=1}^3(P_i)_{jk}P_k,
 \qquad 1\leq i,j\leq3.
 \tag{B2$_{ij}$}
\]
\end{prop}

\begin{proof}
Let \(\alpha=x\times y\) and \(\beta=z\times t\). Since every vector of $\mathbb R^3$ can be written as a cross product, $\alpha$ and $\beta$ are arbitrary.
Substitution of
\[
 [x,y]=A\alpha,\qquad
 [z,t]=A\beta,\qquad
 \langle x,y,w\rangle=P(\alpha)w
\]
into the first Bol identity gives
\[
 P(\alpha)A\beta
 =
 A\bigl(P(\alpha)z\times t+z\times P(\alpha)t\bigr)
 +P(\beta)A\alpha
 +A(A\alpha\times A\beta).
\]
For every \(M\in M_3(\mathbb R)\) and \(z,t\in\mathbb R^3\),
\[
 (Mz)\times t+z\times(Mt)
 =
 \bigl(\operatorname{tr}(M)I-M^T\bigr)(z\times t).
\]
Taking \(M=P(\alpha)\) yields \((\mathrm{B1}_{\mathrm{op}})\).

Since the cyclic sum is an alternating trilinear map, in dimension three
it is enough to evaluate it at \((e_1,e_2,e_3)\). It gives
\[
 P(e_3)e_3+P(e_1)e_1+P(e_2)e_2=0,
\]
which is precisely (C$_{\mathrm{op}}$).

For the second Bol identity, put $D(x,y):=P(x\times y)$,
then
\[
 [D(x,y),D(z,t)]
 =
 D(D(x,y)z,t)+D(z,D(x,y)t).
\]
Using again the preceding vector-product identity gives
\[
 [P(\alpha),P(\beta)]
 =
 P\!\left(
 \bigl(\operatorname{tr}(P(\alpha))I-P(\alpha)^T\bigr)\beta
 \right),
\]
and evaluation at \(\alpha=e_i\), \(\beta=e_j\) gives
(B2$_{\mathrm{ij}}$).
\end{proof}


We give here a technical lemma that will provide us with a very effective tool.

\begin{lem}\label{nlemma} 
Let $B$ be an oriented Euclidean vector space of dimension $3$, let
$A\in \operatorname{GL}(B)$, and let
\[
Q:B\longrightarrow \operatorname{End}(B)
\]
be linear. Put
$\tau(u):=\operatorname{tr}Q(u)$, and
$C_u(v):=u\times v$.
Then $Q$ satisfies the homogeneous first Bol identity
\begin{equation}\label{HB1}
\bigl(Q(u)A-\tau(u)A+AQ(u)^T\bigr)v
=
Q(v)Au
\end{equation}
for $u,v\in B$ if and only if there exists a unique $S\in\operatorname{End}(B)$ such that
\begin{equation}\label{HB2}
Q(u)^T
=
SC_u-\frac12\operatorname{tr}(SC_u)\,I\end{equation}
for every $u\in B$, and
\begin{equation}\label{HB3}
Q(u)Av-Q(v)Au+AS(u\times v)=0
\qquad (u,v\in B).
\end{equation}
Equivalently, after eliminating $Q$ by means of \eqref{HB2}, the homogeneous
first Bol identity is reduced to the single condition (depending on $A$)
\begin{equation}\label{HB4}
-C_uS^TAv+C_vS^TAu
-\frac12\operatorname{tr}(SC_u)\,Av
+\frac12\operatorname{tr}(SC_v)\,Au
+AS(u\times v)=0.\end{equation}
The left-hand side of \eqref{HB4} is alternating in $u,v$; hence it is
enough to impose this identity for
$(u,v)\in\{(e_1,e_2),(e_2,e_3), (e_3,e_1)\}$.

In particular, the universal part of the homogeneous first Bol identity
is parametrized by the nine-dimensional space $\operatorname{End}(B)$. 
\end{lem}

\begin{proof}
Interchanging $u$ and $v$ in \eqref{HB1} and adding the two resulting
identities gives $$
A\Bigl(
Q(u)^Tv+Q(v)^Tu-\tau(u)v-\tau(v)u
\Bigr)=0.$$
Since $A$ is invertible,
\begin{equation}\label{H_0}
Q(u)^Tv+Q(v)^Tu
=
\tau(u)v+\tau(v)u.
\end{equation}
Define
\[
F(u,v):=Q(u)^Tv-\tau(u)v.
\]
Then \eqref{H_0} says precisely that
$F(u,v)=-F(v,u)$,
so $F$ is an alternating bilinear map $B\times B\to B$.

Since $\dim B=3$, the cross product induces an isomorphism
$$
\Lambda^2B\longrightarrow B,
\qquad
u\wedge v\longmapsto u\times v.
$$
Consequently, there exists a unique linear map
$S\in\operatorname{End}(B)$
such that
\[
F(u,v)=S(u\times v).
\]
Thus
\[
Q(u)^Tv-\tau(u)v=S(u\times v)=SC_uv,
\]
for every \(v\in B\), and hence
\begin{equation}\label{HBaux}
Q(u)^T=\tau(u)I+SC_u.
\end{equation}

Taking traces in \eqref{HBaux} yields
\[
\tau(u)
=
3\tau(u)+\operatorname{tr}(SC_u),
\]
and therefore
\[
\tau(u)=-\frac12\operatorname{tr}(SC_u).
\]
Substitution in \eqref{HBaux} gives the claim \eqref{HB2}.

Conversely, suppose that \(Q\) is defined by \eqref{HB2}. Taking traces gives
\[
\tau(u)=-\frac12\operatorname{tr}(SC_u),
\]
and therefore \eqref{HB2} yields
\[
Q(u)^Tv-\tau(u)v=S(u\times v),
\]
which is alternating in \(u,v\). Hence \eqref{H_0} is automatically
satisfied.

It remains to recover the full homogeneous first Bol identity. Writing
\eqref{HB1} as
$$Q(u)Av-Q(v)Au
+
A\bigl(Q(u)^Tv-\tau(u)v\bigr)=0,
$$
and using
$$
Q(u)^Tv-\tau(u)v=S(u\times v),
$$
we obtain
the claim \eqref{HB3}.

Finally, transposing \eqref{HB2} and using $C_u^T=-C_u$, we obtain
$$
Q(u)
=
-C_uS^T-\frac12\operatorname{tr}(SC_u)\,I.
$$
Substitution into \eqref{HB3} gives \eqref{HB4}.

Every term in \eqref{HB4} changes sign when $u$ and $v$ are
interchanged. Hence its left-hand side is alternating, and in dimension
three it is enough to evaluate it on the three basic pairs
$(e_1,e_2)$, $(e_2,e_3)$, and $(e_3,e_1)$.
\end{proof}

\subsubsection{The symmetric cases} \label{sec:2.2.1}

We begin with the two symmetric canonical matrices
$A=I_3$, and $A=\operatorname{diag}(1,1,-1)$.
In both cases the binary product corresponds to a simple real Lie algebra: respectively
\(\mathfrak{so}(3)\) and
\(\mathfrak{so}(2,1)\cong\mathfrak{sl}_2(\mathbb R)\).

\begin{prop}\label{prop:symmetric-cases}
For each of the two symmetric canonical binary products, the compatible
ternary product is uniquely determined and is the canonical Lie triple
product
\[
 \langle x,y,z\rangle=[[x,y],z].
\]
Hence Cases (I.1) and (I.2) yield, respectively,
\(\mathfrak{so}(3)\) and
\(\mathfrak{so}(2,1)\cong\mathfrak{sl}_2(\mathbb R)\), regarded as Bol
algebras of Lie type.
\end{prop}

\begin{proof}
For $A=I_3$ the binary algebra is $\mathfrak{so}(3)$, while for
$A=\operatorname{diag}(1,1,-1)$ it is
$\mathfrak{so}(2,1)\simeq\mathfrak{sl}_2(\mathbb R)$.
For a Lie algebra, the Lie-type ternary product
\[
\langle x,y,z\rangle_0=[[x,y],z]
\]
satisfies the cyclic identity and both Bol identities.
It remains to prove uniqueness.

Let $P_0$ be the corresponding map and let $P$ be any
other solution of the first Bol identity (B1$_{\mathrm{op}}$):
$$ \bigl(
 P(\alpha)A-\operatorname{tr}(P(\alpha))A
 +AP(\alpha)^T
 \bigr)\beta
 =
P(\beta)A\alpha+A(A\alpha\times A\beta).$$
Put
$Q=P-P_0$.
Since in (B1$_{\mathrm{op}}$) the constant terms cancel, $Q$ satisfies the
homogeneous first Bol identity.

By Lemma~\ref{nlemma}, there exists a unique
$S\in\operatorname{End}(B)$ such that
\[
Q(u)^T
=
SC_u-\frac12\operatorname{tr}(SC_u)\,I,
\]
and \(S\) satisfies \eqref{HB4}.
Write
$A=\operatorname{diag}(1,1,\varepsilon)$, with
$\varepsilon\in\{1,-1\}$,
so that the two values of $\varepsilon$ correspond to Cases
(I.1) and (I.2), respectively, and put
$S=(s_{ij})$.
Since the left-hand side of \eqref{HB4} is alternating, it is enough to
evaluate it on the three pairs
$(e_1,e_2)$, $(e_2,e_3)$ and $(e_3,e_1)$.
The corresponding equations are
\[
3s_{13}+s_{31}=0,\qquad
3s_{23}+s_{32}=0,\qquad
-s_{11}-s_{22}+\varepsilon s_{33}=0,
\]
\[
s_{11}-s_{22}-\varepsilon s_{33}=0,\qquad
s_{12}+3s_{21}=0,\qquad
s_{13}+3s_{31}=0,
\]
\[
3s_{12}+s_{21}=0,\qquad
-s_{11}+s_{22}-\varepsilon s_{33}=0,\qquad
s_{23}+3s_{32}=0.
\]
The first and the last equations involving \(s_{13},s_{31}\) give
$s_{13}=s_{31}=0$,
and similarly
$s_{23}=s_{32}=0$, $s_{12}=s_{21}=0$.
The remaining three equations are
\[
-s_{11}-s_{22}+\varepsilon s_{33}=0,\qquad
s_{11}-s_{22}-\varepsilon s_{33}=0,\qquad
-s_{11}+s_{22}-\varepsilon s_{33}=0,
\]
and they imply
$s_{11}=s_{22}=s_{33}=0$.
Hence $S=0$, and the preceding lemma gives $Q=0$. Therefore
$P=P_0$, proving uniqueness.
\end{proof}

\subsubsection{The nonsymmetric block family: common reduction} \label{sec:2.2.2}

We now consider the canonical matrices of type (II), all of which are
contained in the two-parameter block family
\[
 A=A(a,b):=
 \begin{pmatrix}
 a&-1&0\\
 1&b&0\\
 0&0&1
 \end{pmatrix},
 \qquad 1+ab\neq0.
\tag{II}
\]
The individual normal forms correspond to
\[
\begin{array}{c|c}
\text{case}&(a,b)\\ \hline
\mathrm{(II.1)}&(0,0)\\
\mathrm{(II.2^\pm)}&(\pm1,0)\\
\mathrm{(II.3^\pm)}&(\pm s,\pm s),\quad s>0\\
\mathrm{(II.4)}&(s,-s),\quad s>0,\ s\neq1.
\end{array}
\]

For this family, the first stage of the classification is common to all
subcases: one solves the affine linear system
(B1$_{\mathrm{op}}$) and (C$_{\mathrm{op}}$), and only afterwards
uses (B2$_{\mathrm{ij}}$).  This common treatment avoids repeating the
same operator reduction for each canonical representative.

\subsubsection{Case {(II.1)}}

Let \(a=b=0\) in (II), so that
\[
A=
\begin{pmatrix}
0&-1&0\\
1&0&0\\
0&0&1
\end{pmatrix}
=
\begin{pmatrix}
J&0\\
0&1
\end{pmatrix},
\mbox{ with }
J=
\begin{pmatrix}
0&-1\\
1&0
\end{pmatrix}.
\]
Put
$W=\langle e_1,e_2\rangle$
and define
\[
R(u)=P(A^{-1}u).
\]
Since
\[
Ae_1=e_2,\qquad Ae_2=-e_1,\qquad Ae_3=e_3,
\]
we have
\begin{equation}\label{P=AR}
P_1=R(e_2),\qquad
P_2=-R(e_1),\qquad
P_3=R(e_3),
\end{equation}
and we concentrate now on $R$.
Replacing $\alpha$ and $\beta$ in \textup{(B1$_{\operatorname{op}}$)} by $A^{-1}u$ and
$A^{-1}v$, respectively, gives
\begin{equation} \label{N}
R(v)u=
\bigl(R(u)+AR(u)^TA^{-1}
-\operatorname{tr}(R(u))I\bigr)v-A(u\times v).
\end{equation}
Write
\[
R(u)=
\begin{pmatrix}
X(u)&p(u)\\
q(u)^T&r(u)
\end{pmatrix}.
\]
By the linearity of $R$, the maps
\[
X:B\to M_2(\mathbb R),\qquad p,q:B\to W,\qquad r:B\to\mathbb R
\]
are linear.
Since, assuming directly  $X\in M_2(\mathbb R)$, one has
\[
-JX^TJ=\operatorname{tr}(X)I_2-X,
\]
we obtain
\[
R(u)+AR(u)^TA^{-1}-\operatorname{tr}(R(u))I
=
\begin{pmatrix}
-r(u)I_2&p(u)+Jq(u)\\
q(u)^T-p(u)^TJ&r(u)-\operatorname{tr}X(u)
\end{pmatrix}. \nonumber 
\]
If $u=x+se_3$ and $v=y+te_3$, with $x,y\in W$, then
\[
A(u\times v)
=
\binom{tx-sy}{-x^TJy}. \nonumber
\]
Thus \eqref{N} is equivalent to
\begin{equation}\label{3}
X(v)x+s\,p(v)
=(s-r(u))y+t\bigl(p(u)+Jq(u)-x\bigr),
\end{equation}
\begin{equation}\label{4}
q(v)^Tx+s\,r(v)
=
\bigl(q(u)^T-p(u)^TJ\bigr)y
+\bigl(r(u)-\operatorname{tr}X(u)\bigr)t+x^TJy.
\end{equation}
Take first $u=x=x_1e_1+x_2e_2$ and $v=y=y_1e_1+y_2e_2$ in $W$. Thus $s=t=0$, and equation \eqref{3} gives
\begin{equation}
\label{X(y)x}
X(y)x=-r(x)y.
\end{equation}
Let
$r(e_1)=r_1$, $r(e_2)=r_2$,
thus
\[
r(e_1)y=r_1y
=
\begin{pmatrix}
y_1r_1\\
y_2r_1
\end{pmatrix},
\qquad
r(e_2)y=r_2y
=
\begin{pmatrix}
y_1r_2\\
y_2r_2
\end{pmatrix}.
\]
Hence
$$X(y)x=-r(x_1e_1+x_2e_2)y=-x_1(r_1y)-x_2(r_2y)=-x_1\begin{pmatrix}
y_1r_1\\
y_2r_1
\end{pmatrix}-x_2\begin{pmatrix}
y_1r_2\\
y_2r_2
\end{pmatrix}$$
and therefore
\begin{equation}\label{5}
X(y)=-
\begin{pmatrix}
y_1r_1&y_1r_2\\
y_2r_1&y_2r_2
\end{pmatrix},\mbox{ and }\operatorname{tr}X(y)
=-y_1r_1-y_2r_2
=-r(y).
\end{equation}
Alternatively, this also follows, in fact, directly from \eqref{X(y)x}: for fixed $y$, the identity
$X(y)x=-r(x)y$ for every $x\in W$ means that
\[
X(y)=-\,y\otimes r|_W.
\]
Since $\operatorname{tr}(y\otimes r|_W)=r(y)$, we obtain
$\operatorname{tr}X(y)=-r(y)$,
as in \eqref{5}.

Taking $u=v=e_3$ in \eqref{3} and \eqref{4}  gives
\begin{equation}\label{6}
q(e_3)=0,\qquad \operatorname{tr}X(e_3)=0.
\end{equation}
On the other hand, taking $u=x\in W$ and $v=e_3$ in \eqref{4}
gives
\[
r(x)=\operatorname{tr}X(x).
\]
Together with \eqref{5}, this yields
$r(x)=0$ and $X(x)=0$, for
$x\in W$.

Now put
$X=X(e_3)$ and $r_3=r(e_3)$.
Taking $u=e_3$, $v=x\in W$ in \eqref{3} gives
\[
p(x)=(1-r_3)x, \nonumber
\]
while \eqref{4} gives $p(e_3)=0$. Taking instead $u=x\in W$,
$v=e_3$ in \eqref{3}, we obtain
\begin{equation}\label{u9}
Jq(x)=(X+r_3I)x,
\end{equation}
and therefore
\[
q(x)=-J(X+r_3I)x.
\]

Let $Q\in \operatorname{M}_2(\mathbb R)$ be defined by $q(x)=Qx$. Equation
\eqref{4} for $x,y\in W$ now reduces to
\begin{equation}\label{u10}
Q-Q^T=r_3J.
\end{equation}
By \eqref{6}, $\operatorname{tr}X=0$, and therefore $-JX$ is
symmetric. From \eqref{u9},
\[
Q=-JX-r_3J,
\]
so that
\[
Q-Q^T=-2r_3J.
\]
Comparison with \eqref{u10} gives $3r_3J=0$, hence
$r_3=0$.
Consequently
\[
R(x+se_3)=
\begin{pmatrix}
sX&x\\
(-JXx)^T&0
\end{pmatrix}. \nonumber 
\]
Again by \eqref{6}, $\operatorname{tr}X=0$, thus we write
\[
X=
\begin{pmatrix}
-\gamma&\alpha\\
\beta&\gamma
\end{pmatrix}.
\]
By \eqref{P=AR},
we have
\[
P_1=R(e_2),\qquad
P_2=-R(e_1),\qquad
P_3=R(e_3),
\]
hence
\[
P_1=
\begin{pmatrix}
0&0&0\\
0&0&1\\
\gamma&-\alpha&0
\end{pmatrix},
\;
P_2=
\begin{pmatrix}
0&0&-1\\
0&0&0\\
-\beta&-\gamma&0
\end{pmatrix},\;
P_3=
\begin{pmatrix}
-\gamma&\alpha&0\\
\beta&\gamma&0\\
0&0&0
\end{pmatrix},
\tag{P-II.1}
\]
with $\alpha,\beta,\gamma\in\mathbb R$. The cyclic identity is automatically satisfied. A direct substitution
into \textup{(B$2_{ij})$} shows that the second Bol identity is satisfied
identically for all $\alpha,\beta,\gamma$. Thus Case {\rm II.1}
gives a three-parameter family of compatible ternary products.
(Notice that the three-dimensional freedom obtained here agrees with the homogeneous reduction of Lemma \ref{nlemma}, for which the corresponding space of endomorphisms $S$ has dimension $3$.)

\medskip
It remains to determine the residual equivalence. A direct computation
from the symmetric and skew-symmetric parts of $A$ gives
\[
\operatorname{Stab}(A)=
\left\{
g_B=
\begin{pmatrix}
B&0\\
0&1
\end{pmatrix}
:\ B\in\operatorname{SL}_2(\mathbb R)
\right\}.
\]
Under the induced action
\[
P'(u)=g_BP(g_B^Tu)g_B^{-1},
\]
if
\[
B=
\begin{pmatrix}
a&b\\
c&d
\end{pmatrix},
\qquad ad-bc=1,
\]
then the parameters transform according to
\[
\alpha'
 =a^2\alpha+2ab\gamma-b^2\beta,
\]
\[
\beta'
 =-c^2\alpha-2cd\gamma+d^2\beta,
\]
\[
\gamma'
 =ac\alpha+(ad+bc)\gamma-bd\beta.
\]
Equivalently, putting
\[
M_Q=
\begin{pmatrix}
\alpha&\gamma\\
\gamma&-\beta
\end{pmatrix},
\]
one has
$M_Q'\,=\,BM_QB^T$. 
Hence the residual classification is precisely the
$\operatorname{SL}_2(\mathbb R)$-congruence classification of the
real binary quadratic form
\[
Q(X,Y)=\alpha X^2+2\gamma XY-\beta Y^2.
\]
The invariant
\[
D:=\gamma^2+\alpha\beta=-\det M_Q,
\]
together with the inertia of $M_Q$, determines the nondegenerate
orbits, while the rank-one case splits according to sign. Thus one
may choose the normal forms
\begin{equation} \label{SL2II.1}
\begin{array}{c|c}
\text{condition}&(\alpha,\beta,\gamma)\\ \hline
Q=0&(0,0,0)\\
D=0,\ Q\neq0&(1,0,0)\ \text{or}\ (-1,0,0)\\
D>0&(\sqrt D,\sqrt D,0)\\
D<0&(\sqrt{-D},-\sqrt{-D},0)\
      \text{or}\ (-\sqrt{-D},\sqrt{-D},0)
\end{array}, 
\end{equation} 
that we sum up as
$(\alpha,\beta)\in\{(\pm1,0),(a,a),(b,-b)\}$, with $a>0, b\in\mathbb{R}$, and $\gamma=0$.

\subsubsection{Cases {(II.$2^\pm$)}}

Let $(a,b)=(\pm1,0)$ in (II), thus
\[
 A=
 \begin{pmatrix}
 \varepsilon&-1&0\\
 1&0&0\\
 0&0&1
 \end{pmatrix},
 \qquad \varepsilon\in\{1,-1\}.
\]
Here we prove that the affine system
(B1$_{\mathrm{op}}$), (C$_{\mathrm{op}}$) has the one-parameter
solution
\begin{equation} \label{eqP1P2P3} 
 P_1=
 \begin{pmatrix}
 0&0&0\\
 0&0&1\\
 1&-\lambda&0
 \end{pmatrix},
 \;
 P_2=
 \begin{pmatrix}
 0&0&-1\\
 0&0&0\\
 0&-1&0
 \end{pmatrix},\;
 P_3=
 \begin{pmatrix}
 -1&\lambda&0\\
 0&1&0\\
 0&0&0
 \end{pmatrix},\end{equation} 
with $\lambda\in\mathbb R$. Indeed,
let
\[
A=
\begin{pmatrix}
C&0\\
0&1
\end{pmatrix},
\mbox{ with }
C=
\begin{pmatrix}
\varepsilon&-1\\
1&0
\end{pmatrix},
\mbox{ and } \varepsilon\in\{\pm1\}.
\]
Put
$W=\langle e_1,e_2\rangle$
and write, for $u\in\mathbb R^3$,
\[
R(u)=P(A^{-1}u)
=
\begin{pmatrix}
X(u)&p(u)\\
q(u)^T&r(u)
\end{pmatrix},
\]
where $X(u)\in M_2(\mathbb R)$ and $p(u),q(u)\in W$.

As before, the first Bol identity  is \eqref{N}. 
Set
\[
J=
\begin{pmatrix}
0&-1\\
1&0
\end{pmatrix},
\qquad
D:=-CJ=
\begin{pmatrix}
1&\varepsilon\\
0&1
\end{pmatrix}.
\]
For
\[
X=\begin{pmatrix}x_{11}&x_{12}\\x_{21}&x_{22}\end{pmatrix}
\]
one has
$$
X+CX^TC^{-1}-\operatorname{tr}(X)I
=
\begin{pmatrix}
-\varepsilon x_{21}&
x_{21}+\varepsilon(x_{11}-x_{22})\\
0&\varepsilon x_{21}
\end{pmatrix}
=:\Phi(X).
$$
Moreover, if $u=x+se_3$ and $v=y+te_3$, with $x,y\in W$, then
$$
A(u\times v)
=
\binom{tDx-sDy}{-x^TJy}.
$$
We first consider $x,y\in W$. The $W$-component of \eqref{N} is
\begin{equation}
\label{a3}
X(y)x=\bigl(\Phi(X(x))-r(x)I\bigr)y.
\end{equation}
Write
\[
X_i=X(e_i),\qquad r_i=r(e_i),\qquad i=1,2.
\]
Evaluating \eqref{a3} at $(e_i,e_j)$ gives
\begin{equation}
\label{a4}
X_1=
\begin{pmatrix}
-r_1&\varepsilon r_1-r_2\\
0&0
\end{pmatrix},\qquad
X_2=
\begin{pmatrix}
-\varepsilon r_1&-r_1+\varepsilon r_2\\
-r_1&-\varepsilon r_1-r_2
\end{pmatrix}.
\end{equation}

\noindent 
Taking $u=v=e_3$ in \eqref{N} gives
\begin{equation}
\label{a5}
q(e_3)=0,\qquad \operatorname{tr}X(e_3)=0.
\end{equation}
On the other hand, taking $u=x\in W$ and $v=e_3$, the last component
of \eqref{N} gives
\begin{equation}
\label{a6}
r(x)=\operatorname{tr}X(x).
\end{equation}
From \eqref{a4} and \eqref{a6} we obtain successively
$r_1=0$ and $r_2=0$,
and consequently
\begin{equation}
\label{a7}
X(x)=0,\qquad r(x)=0
\end{equation}
for $x\in W$.
Taking now $u=e_3$, $v=x\in W$ in the last component of
\eqref{N}, and using \eqref{a5}--\eqref{a7}, gives
$
p(e_3)=0$.

\medskip
It remains to determine the operators involving $e_3$. Put
\[
X:=X(e_3)=
\begin{pmatrix}
m&n\\
z&-m
\end{pmatrix},
\qquad
\rho:=r(e_3).
\]
There exist matrices $U,V\in M_2(\mathbb R)$ such that
\[
p(x)=Ux,\qquad q(x)=Vx
\qquad (x\in W).
\]
The equations obtained from $(u,v)=(e_3,x)$ and $(x,e_3)$ give
\begin{equation}
\label{a9}
U=\Phi(X)-\rho I+D,
\qquad
V=C^{-1}(X-U+D). \nonumber 
\end{equation}
Explicitly,
\begin{equation}
\label{a10}
U=
\begin{pmatrix}
1-\rho-\varepsilon z&
\varepsilon(2m+1)+z\\
0&1-\rho+\varepsilon z
\end{pmatrix},
\quad
V=
\begin{pmatrix}
z&-m+\rho-\varepsilon z\\
-m-\rho&\varepsilon(m+\rho)-n
\end{pmatrix}.
\end{equation}

\noindent
Finally, the last component of \eqref{N} for $x,y\in W$ is
equivalent to
\begin{equation}
\label{a11}
V-V^T=U^TC^{-1}+J.
\end{equation}
Substituting \eqref{a10} into \eqref{a11} gives
\[\left\{\begin{array}{l}
    3\rho=0\\
2\varepsilon z-3\rho=0\\
\varepsilon(-2m+\rho-2)-2z=0
\end{array}\right.,\mbox{ hence }\left\{\begin{array}{l}
     \rho=0\\ z=0\\ m=-1 
\end{array},\right.
\]
while $n$ remains arbitrary. Writing $n=\lambda$, we obtain
\[
X=
\begin{pmatrix}
-1&\lambda\\
0&1
\end{pmatrix},
\qquad
U=
\begin{pmatrix}
1&-\varepsilon\\
0&1
\end{pmatrix},
\qquad
V=
\begin{pmatrix}
0&1\\
1&-\varepsilon-\lambda
\end{pmatrix}. \nonumber 
\]

\noindent 
Thus
\[
R_1=
\begin{pmatrix}
0&0&1\\
0&0&0\\
0&1&0
\end{pmatrix},
\;
R_2=
\begin{pmatrix}
0&0&-\varepsilon\\
0&0&1\\
1&-\varepsilon-\lambda&0
\end{pmatrix},\;
R_3=
\begin{pmatrix}
-1&\lambda&0\\
0&1&0\\
0&0&0
\end{pmatrix}.
\]
Since
\[
Ae_1=\varepsilon e_1+e_2,\qquad
Ae_2=-e_1,\qquad
Ae_3=e_3,
\]
we have
\[
P_1=\varepsilon R_1+R_2,\qquad
P_2=-R_1,\qquad
P_3=R_3.
\]
Therefore, we obtain \eqref{eqP1P2P3} 
which is independent of the sign $\varepsilon$.
Moreover,
\[
P_1e_1+P_2e_2+P_3e_3=0,
\]
so the cyclic identity is automatically satisfied.

Again, (B2$_{\operatorname{ij}}$) is satisfied identically, so no further
restriction on $\lambda$ occurs by this.

In order to show that we can reduce $\lambda$ to zero, we compute $\operatorname{Stab}(A)$, noticing that it must leave both the symmetric and the skew-symmetric parts $S$ and $K$ of $A$ invariant under the twisted action. Since
$\ker K$ is spanned by $e_3$,
the condition
\[
gKg^T=(\det g)K
\]
implies that
\[
g=\begin{pmatrix}B&u\\0&c\end{pmatrix},
\qquad B\in GL_2(\mathbb R).
\]
Using $BJB^T=(\det B)J$, comparison with
$gKg^T=(\det g)K$ gives $c=1$. On the other hand,
\[
S=\begin{pmatrix}
\varepsilon&0&0\\
0&0&0\\
0&0&1
\end{pmatrix},
\]
and the condition $gSg^T=(\det g)S$ gives $u=0$, $\det B=1$, and
\[
B\begin{pmatrix}\varepsilon&0\\0&0\end{pmatrix}B^T
=
\begin{pmatrix}\varepsilon&0\\0&0\end{pmatrix}.
\]
It follows that
\[
B=\begin{pmatrix}\sigma&t\\0&\sigma\end{pmatrix},
\qquad \sigma\in\{\pm1\},\quad t\in\mathbb R.
\]
Therefore
\[
\operatorname{Stab}(A)=
\left\{
\begin{pmatrix}
\sigma&t&0\\
0&\sigma&0\\
0&0&1
\end{pmatrix}
:\ \sigma=\pm1,\ t\in\mathbb R
\right\}.
\]
In particular, taking $\sigma=1$, the induced action on the family above is
\[
\lambda\longmapsto\lambda+2t.
\]
Hence $t=-\lambda/2$ reduces every member of the family to $\lambda=0$.

Thus each of the two non-equivalent binary products
(II.2\(^{+}\)) and (II.2\(^{-}\)) admits, up to the stabilizer
of \(A\), a unique compatible ternary product, represented by
\[
 P_1=
 \begin{pmatrix}
 0&0&0\\
 0&0&1\\
 1&0&0
 \end{pmatrix},
 \quad
 P_2=
 \begin{pmatrix}
 0&0&-1\\
 0&0&0\\
 0&-1&0
 \end{pmatrix},
 \quad
 P_3=
 \begin{pmatrix}
 -1&0&0\\
 0&1&0\\
 0&0&0
 \end{pmatrix}.
 \tag{P-II.2}
\]

\subsubsection{Cases {(II.$3^\pm$)}} \label{sec.2.2.5}


Let $(a,b)=(\pm s,\pm s)$ (with $s>0$)
in (II), thus
\[
 A=
 \begin{pmatrix}
 \varepsilon s&-1&0\\
 1&\varepsilon s&0\\
 0&0&1
 \end{pmatrix},
 \qquad s>0,\qquad \varepsilon\in\{1,-1\}.
\]
Put \(c=\varepsilon s\), so that \(c\neq0\).

We first show that the first Bol identity has at most one solution.
Suppose that \(P\) and \(\widetilde P\) are two solutions of
\((B1_{\mathrm{op}})\), and put \(Q=P-\widetilde P\).
By Lemma \ref{nlemma}, there exists a unique
\(S=(s_{ij})\in\operatorname{End}(B)\) satisfying \eqref{HB4}.

Evaluating \eqref{HB4} on
\((e_1,e_2)\), \((e_2,e_3)\), and \((e_3,e_1)\)
gives, among the resulting equations,
\[
s_{12}+3s_{21}=0,\qquad
3s_{12}+s_{21}=0,
\]
and hence
$s_{12}=s_{21}=0$.
Moreover,
\[
s_{13}+3s_{31}=0,\qquad
s_{23}+3s_{32}=0,
\]
while the first two equations corresponding to \((e_1,e_2)\) give
\[
3cs_{13}+cs_{31}+s_{23}-s_{32}=0,
\]
\[
3cs_{23}+cs_{32}-s_{13}+s_{31}=0.
\]
Thus
$s_{13}=-3s_{31}$, and $s_{23}=-3s_{32}$,
and the last two equations reduce to
\[
2cs_{31}+s_{32}=0,\qquad
s_{31}-2cs_{32}=0.
\]
Consequently
\[
(1+4c^2)s_{31}=0,
\]
thus we obtain
$s_{31}=s_{32}=s_{13}=s_{23}=0$.

The remaining equations are
$$-c(s_{11}+s_{22})+s_{33}=0,
\qquad
c(s_{11}-s_{22})-s_{33}=0,$$
$$-c(s_{11}-s_{22})-s_{33}=0.$$
Hence \(s_{33}=0\), \(s_{11}=s_{22}\), and, since \(c\neq0\), we obtain
$s_{11}=s_{22}=0$.
Therefore \(S=0\), and Lemma \ref{nlemma} gives $Q=0$.
Thus (B1$_{\mathrm{op}}$) has at most one solution.

Direct substitution shows that a solution is obtained by setting
\(R(u)=P(A^{-1}u)\) and
\[
R_1=
\begin{pmatrix}
0&0&-1\\
0&0&-c\\
3/c&1&0
\end{pmatrix},
\qquad
R_2=
\begin{pmatrix}
0&0&c\\
0&0&-1\\
-1&3/c&0
\end{pmatrix},
\]
\[
R_3=
\begin{pmatrix}
0&-c-3/c&0\\
c+3/c&0&0\\
0&0&2
\end{pmatrix}.
\]
By uniqueness, this is the unique solution of the first Bol identity.

Since
\[
Ae_1=ce_1+e_2,\qquad
Ae_2=-e_1+ce_2,\qquad
Ae_3=e_3,
\]
we have
\[
P_1=cR_1+R_2,\qquad
P_2=-R_1+cR_2,\qquad
P_3=R_3.
\]
Hence
$P_1e_1+P_2e_2+P_3e_3
=
6e_3\neq0$.
Thus the unique solution of the first Bol identity fails the cyclic
identity. Therefore neither Case {\rm (II.3)$^+$} nor
Case {\rm (II.3)$^-$} admits a compatible ternary product.

\subsubsection{Case {\rm (II.4)}}

Let $(a,b)=(s,-s)$ (with $s>0$, $s\neq1$) in (II), thus
\[
A=
\begin{pmatrix}
s&-1&0\\
1&-s&0\\
0&0&1
\end{pmatrix},
\qquad s>0,\qquad s\neq1.
\]
We first determine the homogeneous freedom in the first Bol identity.
Suppose that \(P\) and \(\widetilde P\) are two solutions of
\((B1_{\mathrm{op}})\), and put
\[
Q=P-\widetilde P.
\]
By Lemma~\ref{nlemma}, there exists a unique
\(S=(s_{ij})\in\operatorname{End}(B)\) satisfying \eqref{HB4}.
Evaluating \eqref{HB4} on
$(e_1,e_2)$, $(e_2,e_3)$, $(e_3,e_1)$
gives, after multiplication by $2$, respectively,
\begin{equation}\label{c1}
\begin{cases}
3ss_{13}+ss_{31}+s_{23}-s_{32}=0,\\
-3ss_{23}-ss_{32}-s_{13}+s_{31}=0,\\
-2ss_{11}+2ss_{22}+2s_{12}-2s_{21}+2s_{33}=0,
\end{cases}
\end{equation}
\begin{equation}\label{c2}
\begin{cases}
2ss_{11}+2ss_{22}+s_{12}-s_{21}-2s_{33}=0,\\
-s(s_{12}+3s_{21})=0,\\
s_{13}+3s_{31}=0,
\end{cases}\end{equation}
and
\begin{equation}\label{c3}
\begin{cases}
s(3s_{12}+s_{21})=0,\\
-2ss_{11}-2ss_{22}+s_{12}-s_{21}-2s_{33}=0,\\
s_{23}+3s_{32}=0.
\end{cases}\end{equation}
Since $s>0$, the second equation of \eqref{c2} and the first equation of
\eqref{c3} give
$$
s_{12}+3s_{21}=0,\qquad
3s_{12}+s_{21}=0,$$
and hence
$s_{12}=s_{21}=0$.
The remaining diagonal equations become
$$-ss_{11}+ss_{22}+s_{33}=0,
\qquad
ss_{11}+ss_{22}-s_{33}=0,
\qquad
-ss_{11}-ss_{22}-s_{33}=0.
$$
The last two imply
$s_{33}=0$, and $s_{11}+s_{22}=0$,
while the first gives $s_{11}=s_{22}$. Therefore
$s_{11}=s_{22}=s_{33}=0$.
Moreover,
$$s_{13}+3s_{31}=0,\qquad
s_{23}+3s_{32}=0,
$$
so
$
s_{13}=-3s_{31}$, and
$s_{23}=-3s_{32}$.
The first two equations in \eqref{c1} then reduce to
$$
2ss_{31}+s_{32}=0,\qquad
s_{31}+2ss_{32}=0.
$$
Consequently
$$
(1-4s^2)s_{31}=0.
$$
Hence, if \(s\neq\frac12\), then
$
s_{31}=s_{32}=s_{13}=s_{23}=0$,
and therefore $S=0$. By Lemma \ref{nlemma}, \(Q=0\), so
(B1$_{\mathrm{op}}$) has at most one solution.

Direct substitution shows that, for $s\neq\frac12$, a solution is obtained by setting
$R(u)=P(A^{-1}u)$ and
$$
R_1=
\begin{pmatrix}
0&0&-1\\
0&0&s\\
3/s&1&0
\end{pmatrix},
\qquad
R_2=
\begin{pmatrix}
0&0&s\\
0&0&-1\\
-1&-3/s&0
\end{pmatrix},\qquad
R_3=
\begin{pmatrix}
0&3/s-s&0\\
3/s-s&0&0\\
0&0&2
\end{pmatrix}.
$$
By uniqueness, this is the unique solution of the first Bol identity
for $s\neq\frac12$. Since
$$
Ae_1=se_1+e_2,\qquad
Ae_2=-e_1-se_2,\qquad
Ae_3=e_3,
$$
we have
$$
P_1=sR_1+R_2,\qquad
P_2=-R_1-sR_2,\qquad
P_3=R_3.
$$
It follows immediately that
$P_1e_1+P_2e_2+P_3e_3=6e_3\neq0$.

Thus the cyclic identity is impossible whenever
$s\neq\frac12$.

\medskip
It remains to consider the exceptional value $s=\frac12$.
In this case the preceding homogeneous system has a one-dimensional
solution space. Indeed, the diagonal and $(s_{12},s_{21})$-equations
still give
$$
s_{11}=s_{12}=s_{21}=s_{22}=s_{33}=0.$$
Moreover,
$$
s_{13}=-3s_{31},\qquad
s_{23}=-3s_{32},$$
and the two remaining equations reduce to
$s_{32}=-s_{31}$.
Thus, writing $s_{31}=\mu$,$$
s_{32}=-\mu,\qquad
s_{13}=-3\mu,\qquad
s_{23}=3\mu,$$
so the homogeneous solution space is one-dimensional.

Direct substitution shows that the affine first Bol identity has the
one-parameter family
$$R_1=
\begin{pmatrix}
0&-4\lambda&-1\\
-2\lambda&-2\lambda&1/2\\
6&1&-2\lambda
\end{pmatrix},
\qquad
R_2=
\begin{pmatrix}
2\lambda&2\lambda&1/2\\
4\lambda&0&-1\\
-1&-6&2\lambda
\end{pmatrix},\qquad
R_3=
\begin{pmatrix}
0&11/2&\lambda\\
11/2&0&\lambda\\
0&0&2
\end{pmatrix}.$$
Since the homogeneous solution space is one-dimensional, this family
exhausts all solutions of (B1$_{\mathrm{op}}$).
From
$$
P_1=\frac12R_1+R_2,\qquad
P_2=-R_1-\frac12R_2,\qquad
P_3=R_3,
$$
we obtain now
$$
P_1e_1+P_2e_2+P_3e_3
=
6\lambda e_1+6\lambda e_2+6e_3,
$$
which is never zero. Therefore the cyclic identity is incompatible with
the first Bol identity for every $s>0$, $s\neq1$.
Hence Case {\rm (II.4)} admits no compatible ternary product.

\subsubsection{The case (III.1)} 

We finally consider
\[
A=
\begin{pmatrix}
0&-1&1\\
1&1&0\\
1&0&0
\end{pmatrix}.
\;\mbox{ noting that }\;
A^{-1}=
\begin{pmatrix}
0&0&1\\
0&1&-1\\
1&1&-1
\end{pmatrix}.
\]

We first show that the first Bol identity determines the ternary product uniquely. 

Suppose that \(P\) and \(\widetilde P\) are two solutions of
\((B1_{\mathrm{op}})\), and put
$Q=P-\widetilde P$.
Then $Q$ satisfies the homogeneous first Bol identity. By 
Lemma \ref{nlemma}, there exists a unique $S\in\operatorname{End}(B)$ such that
\[
Q(u)^T
=
SC_u-\frac12\operatorname{tr}(SC_u)\,I,
\]
and \(S\) satisfies
\begin{equation}\label{C_u}
-C_uS^TAv+C_vS^TAu
-\frac12\operatorname{tr}(SC_u)\,Av
+\frac12\operatorname{tr}(SC_v)\,Au
+AS(u\times v)=0.\end{equation}
Since the left-hand side is alternating in \(u,v\), it is enough to
consider the three pairs
$(e_1,e_2)$, $(e_1,e_3)$ and $(e_2,e_3)$.

Writing
\[
S=
\begin{pmatrix}
s_{11}&s_{12}&s_{13}\\
s_{21}&s_{22}&s_{23}\\
s_{31}&s_{32}&s_{33}
\end{pmatrix},
\]
the three instances of \eqref{C_u} give respectively
\begin{equation}\label{n1}
\begin{cases}
s_{23}-s_{32}+4s_{33}=0,\\
-s_{13}+3s_{23}+s_{31}+s_{32}=0,\\
2s_{12}+s_{13}-2s_{21}-2s_{22}-s_{31}=0,
\end{cases}\end{equation}
\begin{equation}\label{n2}
\begin{cases}
-s_{23}-3s_{32}=0,\\
-s_{12}+2s_{13}+s_{21}-2s_{22}+2s_{31}=0,\\
-3s_{12}-s_{21}=0,
\end{cases}\end{equation}
and
\begin{equation}\label{n3}
\begin{cases}
s_{12}-s_{13}-s_{21}-2s_{22}+s_{31}=0,\\
s_{12}+3s_{21}=0,\\
4s_{11}=0.
\end{cases}\end{equation}

The last equation of \eqref{n2} and the second equation of \eqref{n3} yield
$s_{12}=s_{21}=0$.
The first equation of \eqref{n3} then gives
$s_{31}=s_{13}+2s_{22}$,
whereas the second equation of \eqref{n2} gives
$s_{13}-s_{22}+s_{31}=0$, hence
$2s_{13}+s_{22}=0$.
The third equation of \eqref{n1}, together with $s_{31}=s_{13}+2s_{22}$, gives
$-4s_{22}=0$, thus
$s_{22}=s_{13}=s_{31}=0$.

The first equation of \eqref{n2} gives
$s_{23}=-3s_{32}$,
while the second equation of \eqref{n1} now gives
$3s_{23}+s_{32}=0$,
therefore
$s_{23}=s_{32}=0$.
The first equation of \eqref{n1} then yields $s_{33}=0$, and the last
equation of \eqref{n3} yields $s_{11}=0$, hence
$S=0$.
By the preceding lemma, $Q=0$. Thus \((B1_{\mathrm{op}})\) has at most
one solution.

It remains only to exhibit it. Direct substitution shows that
$$
P_1=
\begin{pmatrix}
2&0&0\\
1&0&2\\
0&-1&0
\end{pmatrix},
\qquad
P_2=
\begin{pmatrix}
-1&2&-2\\
0&0&0\\
0&0&1
\end{pmatrix},
\qquad
P_3=
\begin{pmatrix}
0&1&2\\
0&0&-1\\
0&0&0
\end{pmatrix}
$$
satisfy \((B1_{\mathrm{op}})\), and hence constitute its unique solution.

However,
$P_1e_1+P_2e_2+P_3e_3=6e_1\neq0$,
thus the cyclic identity is incompatible with the first Bol identity.
Therefore Case \((\mathrm{III}.1)\) admits no compatible ternary product.

\subsubsection{Classification for {$[B,B]=B$}}

Combining the preceding arguments gives the classification of all compatible
ternary products for the nine canonical binary products.

\begin{theor}
\label{thm:rank-three-classification}
Let \(B\) be a three-dimensional real Bol algebra such that
$[B,B]=B$.
After a change of basis, its binary product is one of the nine canonical
products of Theorem~\ref{thm:binary-normal-forms}.  The compatible ternary
products are as follows.

\begin{enumerate}
\item
For the symmetric cases \rm{(I.1)} and \rm{(I.2)}, the ternary product is
uniquely
\[
 \langle x,y,z\rangle=[[x,y],z].
\]

\item
For Case \rm{(II.1)}, the compatible ternary products form the
one-parameter family given in Table \eqref{SL2II.1}.  

\item
For each of the two cases \rm{(II.2\(^{\pm}\))}, the compatible ternary
product is unique up to equivalence (see (\rm{P-II.2})).

\item
The cases \rm{(II.3\(^{\pm}\))}, \rm{(II.4)}, and \rm{(III.1)} admit no
compatible ternary product.
\end{enumerate}
\end{theor}


\section{Simplicity questions} \label{sec:simplicity}
A natural question arising from the classification is whether simplicity can be
detected from the binary product.  In Lie theory, every simple algebra is
perfect.  For Bol algebras the situation is subtler, because the notion of
ideal involves both the binary and the ternary products \cite{kuzmin}.  We therefore study
the relation between simplicity and the subspace
\[
 [B,B]=\operatorname{Span}\{[x,y]\mid x,y\in B\}.
\]
Indeed, we recall (see Proposition \ref{simpleBBB}) that if $[B,B]=B$, then $B$ is simple.

We will see that the converse is in general false: the two Bol algebras $B$ given below in Theorem \ref{thm:rank-two} are simple, though $\operatorname{dim}[B,B]=2$, and the four Lie triple systems $B$ given below in Theorem \ref{thm:simpleLTS} are simple, though $\operatorname{dim}[B,B]=0$. However, these are the only simple three-dimensional real Bol algebras with $[B,B]\ne B$.

\subsection{Operator form of the Bol identities}

We shall repeatedly use the following elementary criterion.

\begin{lem} \label{bolideal} 
Let $I\subseteq B$ be a linear subspace. Then $I$ is a Bol ideal if and only if
\[
[I,B]\subseteq I, 
\qquad
P(I\times B)B\subseteq I. 
\]
\end{lem}

\begin{proof}
The first condition is precisely the condition
$
[I,B]\subseteq I$ 
for the binary product. 
The second 
condition is equivalent to
$ 
\langle I,B,B \rangle \subseteq I$
(see \cite{Lister}). By skew-symmetry of the ternary product in its first two variables
gives
\[
\langle B,I,B \rangle \subseteq I
\iff
\langle I,B,B \rangle \subseteq I.
\]
Finally, if \(z\in I\), the cyclic identity gives
\[
\langle x,y,z\rangle
=
-\langle y,z,x\rangle-\langle z,x,y\rangle\in I,
\]
and hence
$\langle B,B,I \rangle \subseteq I$.

Hence these conditions are exactly the defining conditions for a Bol ideal.
\end{proof}

\subsection{The rank-two case} \label{sec:3.2} 

Assume throughout this subsection that
$\operatorname{rank}A=2$.


We first reduce the binary product to a convenient list of normal-form
families.

As observed in Lemma~\ref{lem:twisted-congruence}, the argument relating
the twisted action \eqref{eq:action}
to ordinary congruence does not use the invertibility of $A$.
Therefore, for arbitrary \(A,A'\in M_3(\mathbb R)\),
\[
A'=(\det g)^{-1}gAg^T
\]
for some \(g\in\operatorname{GL}_3(\mathbb R)\) if and only if \(A'\)
is congruent to \(A\) or to \(-A\).

We may consequently apply the real congruence theorem of
Lee--Weinberg \cite[Theorem~II]{Lee} to matrices of rank two.
In dimension three there are only four possible block patterns.

If the canonical form is a sum of three one-dimensional blocks, then
exactly one of them is zero and, up to permutation and simultaneous
change of sign, we obtain
\[
A_{S,+}=
\begin{pmatrix}
1&0&0\\
0&1&0\\
0&0&0
\end{pmatrix},
\qquad
A_{S,-}=
\begin{pmatrix}
1&0&0\\
0&-1&0\\
0&0&0
\end{pmatrix}.
\]

A rank-one two-dimensional indecomposable block, together with a
nonzero one-dimensional block, gives, up to congruence,
\[
A_{N_1}=
\begin{pmatrix}
0&1&0\\
0&0&0\\
0&0&1
\end{pmatrix}.
\]

If instead the two-dimensional block is nonsingular and the remaining
one-dimensional block is zero, the relevant Lee--Weinberg
two-dimensional canonical blocks may all be collected, for the present
purposes, in the single family
\[
A_{N_2}(a,b)=
\begin{pmatrix}
a&-1&0\\
1&b&0\\
0&0&0
\end{pmatrix},
\qquad
ab+1\neq0.
\]
This parametrization is deliberately redundant: different pairs
\((a,b)\) may represent different Lee--Weinberg canonical block types,
and no irredundancy is needed in the argument below.

Finally, the only indecomposable three-dimensional Lee--Weinberg block
of rank two is the \(e=1\) specialization of \(m'_3\). It is congruent
to
\[
A_{N_3}=
\begin{pmatrix}
0&0&-1\\
1&0&0\\
0&0&0
\end{pmatrix}.
\]

Hence every rank-two matrix is equivalent under the twisted action to
a matrix belonging to one of the five families
\[
A_{S,+},\qquad
A_{S,-},\qquad
A_{N_1},\qquad
A_{N_2}(a,b),\qquad
A_{N_3}.
\]
The list is complete but is not intended to be irredundant.

We now determine which of these binary products can occur in a simple
Bol algebra.

\subsubsection{The symmetric cases}

Let
$A=A_{S,\varepsilon}
=
\operatorname{diag}(1,\varepsilon,0)$,
where $\varepsilon=\pm1$,
and put
$W=\operatorname{im}A
=
\operatorname{Span}\{e_1,e_2\}$, thus
$[W,B]\subseteq W$.

Solving \rm{(B1$_{\mathrm{op}}$)} together with the cyclic identity gives, for both
values of $\varepsilon$,
\[
P_1=
\begin{pmatrix}
-\dfrac{2u}{3}&0&\alpha\\[3mm]
-\dfrac{v}{3}&-\dfrac{u}{3}&\beta\\[3mm]
0&0&\dfrac{u}{3}
\end{pmatrix},\qquad
P_2=
\begin{pmatrix}
-\dfrac{v}{3}&-\dfrac{u}{3}&\gamma\\[3mm]
0&-\dfrac{2v}{3}&\delta\\[3mm]
0&0&\dfrac{v}{3}
\end{pmatrix},
\qquad
P_3=
\begin{pmatrix}
0&0&u\\
0&0&v\\
0&0&0
\end{pmatrix}.
\]
The $(1,3)$-entry of \rm{(B2$_{ij}$)} for $(i,j)=(3,1)$ and the
$(2,3)$-entry for $(i,j)=(3,2)$ give
\[
2u^2=0,
\qquad
2v^2=0,
\]
thus
$u=v=0$.
Consequently
$
P_i(B)\subseteq W$,
for  $i=1,2,3$.
Since
$W\times B=B$,
it follows that
\[ 
P(W\times B)B=P(B)B\subseteq W.
\]
By Lemma~\ref{bolideal}, $W$ is a proper Bol ideal, and we obtain the following:

\begin{prop}\label{prop:rank2-symmetric}
No simple $3$-dimensional Bol algebra has a symmetric binary matrix.
\end{prop}

\subsubsection{The case $A_{N_3}$}

Let
\[
A=A_{N_3}=
\begin{pmatrix}
0&0&-1\\
1&0&0\\
0&0&0
\end{pmatrix}.
\]
Putting
$L=\mathbb Re_1$,
the binary products give
$[L,B]\subseteq L$.

In this case the second Bol identity is not needed.  Indeed,
\rm{(B1$_{\mathrm{op}}$)} together with cyclicity already gives
$
P_2(B)\subseteq L$, and 
$P_3(B)\subseteq L$. 
For completeness, the corresponding solution is
\[
P_1=
\begin{pmatrix}
1-r+s&p&0\\
0&r-1&1-r+2s\\
0&q&1-r
\end{pmatrix},
\; P_2=
\begin{pmatrix}
0&r-2s-1&r-2s-1\\
0&0&0\\
0&0&0
\end{pmatrix},
\;
P_3=
\begin{pmatrix}
0&r&s\\
0&0&0\\
0&0&0
\end{pmatrix}.
\]
Since
$L\times B=\operatorname{Span}\{e_2,e_3\}$,
we obtain $
P(L\times B)B\subseteq L$. 
Hence $L$ is a proper Bol ideal (see Lemma \ref{bolideal}). Again, we can state the following:

\begin{prop}\label{prop:rank2-N3}
No simple three-dimensional Bol algebra has binary matrix $A_{N_3}$.
\end{prop}

\subsubsection{The case $A_{N_2}(a,b)$}

Let
\[
A=A_{N_2}(a,b)
=
\begin{pmatrix}
a&-1&0\\
1&b&0\\
0&0&0
\end{pmatrix},
\qquad ab+1\neq0,
\]
and set
$W=\operatorname{Im}A
=
\operatorname{Span}\{e_1,e_2\}$.
Then
$[W,B]\subseteq W$.

The computation is best organized separating the case where $9ab+1$ is zero.

\paragraph{The subcase $9ab+1\neq0$.}

Solving \rm{(B1$_{\mathrm{op}}$)} and cyclicity shows that there exist $u,v\in
\mathbb R$ such that
$(P_i)_{31}=(P_i)_{32}=0$, for $i=1,2,3$,
and
\[
(P_1)_{33}
=
\frac{(3ab-1)u+4av}{9ab+1},
\qquad
(P_2)_{33}
=
\frac{-4bu+(3ab-1)v}{9ab+1},
\qquad
(P_3)_{33}=0,
\]
while
$
P_3(B)\subseteq W$. 

Assume first that $ab\neq0$.  Among the consequences of the second
Bol identity are
\[
av^2+bu^2=0,
\qquad
a\,u(3bu+v)=0,
\qquad
b\,v(u-3av)=0.
\]
If $u=0$, the first equation gives $v=0$, and similarly $v=0$
implies $u=0$.  Hence, if $(u,v)\neq(0,0)$, then
\[
v=-3bu,
\qquad
u=3av.
\]
Substitution yields
\[
(9ab+1)u=(9ab+1)v=0,
\]
contrary to $9ab+1\neq0$.  Therefore
$u=v=0$.
It follows that $ 
P_i(B)\subseteq W$ for $i=1,2,3$.
Since
$W\times B=B$,
the plane $W$ is a proper Bol ideal.

\medskip It remains, within the subcase $9ab+1\ne0$, to take $ab=0$.

Suppose $a=0$ and $b\neq0$.  The second Bol identity gives $u=0$.
If $v=0$, the preceding argument again makes $W$ a Bol ideal.  If
$v\neq0$, the remaining equations reduce the relevant operators to
\[
P_1=
\begin{pmatrix}
0&0&0\\
-v&0&\lambda\\
0&0&0
\end{pmatrix},
\qquad
P_2=
\begin{pmatrix}
v&0&-\lambda\\
0&0&0\\
0&0&-v
\end{pmatrix},
\qquad
P_3=
\begin{pmatrix}
0&0&0\\
0&0&v\\
0&0&0
\end{pmatrix}.
\]
Then
$L=\mathbb Re_2$
is a proper Bol ideal.  Indeed, $ 
[L,B]\subseteq L$, 
and, since
$L\times B=\operatorname{Span}\{e_1,e_3\}$,
the matrices $P_1$ and $P_3$ give
$
P(L\times B)B\subseteq L$.

Suppose now $b=0$ and $a\neq0$.  The second Bol identity gives
$v=0$.
If $u=0$, the plane $W$ is an ideal.  If $u\neq0$, the remaining
operators are
\[
P_1=
\begin{pmatrix}
0&0&0\\
0&u&\lambda\\
0&0&-u
\end{pmatrix},
\qquad
P_2=
\begin{pmatrix}
0&-u&-\lambda\\
0&0&0\\
0&0&0
\end{pmatrix},
\qquad
P_3=
\begin{pmatrix}
0&0&u\\
0&0&0\\
0&0&0
\end{pmatrix}.
\]
Hence
$L=\mathbb Re_1$
is a proper Bol ideal.

Finally, let
$a=b=0$.
If $u=v=0$, the plane $W$ is again a Bol ideal.  

Suppose then
$(u,v)\neq(0,0)$.
The remaining equations give
\[
P_1=
\begin{pmatrix}
0&0&0\\
-v&u&\lambda\\
0&0&-u
\end{pmatrix},
\qquad
P_2=
\begin{pmatrix}
v&-u&-\lambda\\
0&0&0\\
0&0&-v
\end{pmatrix},
\qquad
P_3=
\begin{pmatrix}
0&0&u\\
0&0&v\\
0&0&0
\end{pmatrix}.
\]
Put
$L=\mathbb R(ue_1+ve_2)$.
One checks directly that
$[L,B]\subseteq L$.
Moreover,
\[
(ue_1+ve_2)\times B
=
\operatorname{Span}\{e_3,\,ve_1-ue_2\}.
\]
We have
$P_3(B)\subseteq L$,
and
\[
P(ve_1-ue_2)
=
vP_1-uP_2
=
\begin{pmatrix}
-uv&u^2&\lambda u\\
-v^2&uv&\lambda v\\
0&0&0
\end{pmatrix},
\]
whose image is contained in $L$.  Hence $L$ is a proper Bol ideal.

\paragraph{The subcase $9ab+1=0$.}

Assume now $9ab+1=0$,
then $a,b\neq0$, and 
$b=-\frac1{9a}$.

If we put
$w=-3ae_1+e_2$,
then
\[
[w,e_1]=[w,e_2]=0,
\qquad
[w,e_3]=\frac23 w.
\]
Thus
$L=\mathbb R w$
is a one-dimensional ideal of the underlying anticommutative algebra.

Solving \rm{(B1$_{\mathrm{op}}$)} and cyclicity on the exceptional locus gives
parameters $q,v$ for which
\[
P_3=
\begin{pmatrix}
0&0&3av\\
0&0&v\\
0&0&0
\end{pmatrix}.
\]
Two consequences of \rm{(B2$_{ij}$)} are
$ 
av(3q+v)=0$ and 
$a^2v(3q-v)=0$. 
Since $a\neq0$, these imply
$v=0$.

If $q=0$, then
$W=\operatorname{Span}\{e_1,e_2\}$
is a proper Bol ideal.  Suppose therefore that $q\neq0$.
Furthermore,
\[
w\times e_1=-e_3,
\qquad
w\times e_2=-3ae_3,
\qquad
w\times e_3=e_1+3ae_2.
\]
Since $P_3=0$, it remains only to consider
\[
P(e_1+3ae_2)=P_1+3aP_2.
\]
Writing the remaining free coefficients as
$c_1,c_2,c_3,c_4$, the nontrivial equations from the second Bol
identity reduce to
\[
3ac_2+12ac_3+5c_1=0,
\]
\[
15ac_4+4c_2+c_3=0.
\]
Moreover,
\[
P_1+3aP_2=
\begin{pmatrix}
-\dfrac{3aq}{2}&-\dfrac{9a^2q}{2}&c_1+3ac_3\\[2mm]
\dfrac q2&\dfrac{3aq}{2}&c_2+3ac_4\\[2mm]
0&0&0
\end{pmatrix}.
\]
The two displayed relations imply
$ 
c_1+3ac_3
=
-3a(c_2+3ac_4)$.
Hence every column of $P_1+3aP_2$ is a multiple of
$-3ae_1+e_2$.
Therefore
$P(L\times B)B\subseteq L$,
and $L$ is a proper Bol ideal. Summarizing, we have proved the following:

\begin{prop}\label{prop:rank2-N2}
No simple three-dimensional Bol algebra has binary matrix
$A_{N_2}(a,b)$.
\end{prop}

\subsubsection{The case $A_{N_1}$}

It remains to consider
\[
A=A_{N_1}
=
\begin{pmatrix}
0&1&0\\
0&0&0\\
0&0&1
\end{pmatrix}.
\]
The binary products are therefore
\[
[e_1,e_2]=e_3,
\qquad
[e_2,e_3]=0,
\qquad
[e_3,e_1]=e_1,
\]
so that
$[B,B]=\operatorname{Span}\{e_1,e_3\}$.

We first determine the possible proper ideals of the underlying
anticommutative algebra.

\begin{lem}\label{lem:AN1-binary-ideals}
The underlying anticommutative algebra has no one-dimensional ideals.
\end{lem}

\begin{proof}
For
$v=xe_1+ye_2+ze_3$,
one has
\[
[v,e_1]=ze_1-ye_3,
\qquad
[v,e_2]=xe_3,
\qquad
[v,e_3]=-xe_1.
\]
If $\mathbb Rv$ were an ideal, these three vectors would all be
proportional to $v$,  forcing
$x=y=z=0$.
\end{proof}

\begin{co}\label{cor:AN1-only-ideal}
The only possible nonzero Bol ideal, if one exists, is
\[
I=[B,B]=\operatorname{Span}\{e_1,e_3\}.
\]
\end{co}

\begin{proof}
By Lemma~\ref{lem:AN1-binary-ideals}, such an ideal must have
dimension two.  The quotient is then one-dimensional, hence has zero
binary product.  Therefore
$[B,B]\subseteq I$,
and equality follows from
$\dim[B,B]=2$.
\end{proof}

Thus simplicity is reduced to deciding whether this single plane $I$
is a Bol ideal.
Solving \rm{(B1$_{\mathrm{op}}$)} together with cyclicity gives
\[
P_1=0,
\qquad
P_2=
\begin{pmatrix}
-\dfrac{2t}{3}&0&0\\[1mm]
\alpha&\dfrac t3&0\\[1mm]
\beta&1&\dfrac{2t}{3}
\end{pmatrix},
\qquad
P_3=
\begin{pmatrix}
2&0&0\\[1mm]
0&-2&-\dfrac t3\\[1mm]
t&0&-1
\end{pmatrix}.
\]
The second Bol identity is equivalent to
\[
15\alpha+\beta t=0,
\qquad
12\beta+5t^2=0.
\]
Consequently
$\alpha=\frac{t^3}{36}$, 
$\beta=-\frac{5t^2}{12}$.
Thus every compatible ternary product is given by
\[
P_1(t)=0,
\qquad
P_2(t)=
\begin{pmatrix}
-\dfrac{2t}{3}&0&0\\[3mm]
\dfrac{t^3}{36}&\dfrac t3&0\\[3mm]
-\dfrac{5t^2}{12}&1&\dfrac{2t}{3}
\end{pmatrix},
\qquad
P_3(t)=
\begin{pmatrix}
2&0&0\\[1mm]
0&-2&-\dfrac t3\\[1mm]
t&0&-1
\end{pmatrix},
\qquad t\in\mathbb R.
\]

Every algebra in this family is simple.  Indeed, by Corollary \ref{cor:AN1-only-ideal}, the only possible
proper Bol ideal is
$I=\operatorname{Span}\{e_1,e_3\}$.

Since
$e_3\times e_1=e_2$ and $e_1\times e_2=e_3$,
we have both $e_2,e_3\in I\times B$. If $t\neq0$, then
$$
P_2(t)e_2=\frac{t}{3}e_2+e_3\notin I,$$
whereas, if $t=0$, then
$$
P_3(0)e_2=-2e_2\notin I,$$
so in either case 
$P(I\times B)B\nsubseteq I$. 


In either case $I$ is not a Bol ideal, and every Bol algebra in the family is simple.

\medskip
It remains to determine the parameter $t$ up to isomorphism.
The right radical of $A_{N_1}$ is $\mathbb Re_1$, whereas its left
radical is $\mathbb Re_2$.  Hence, if
\[
gA_{N_1}g^T=(\det g)A_{N_1},
\]
then the corresponding coordinate lines are preserved.  Thus $g$ has
the form
\[
g=
\begin{pmatrix}
x&0&0\\
0&y&0\\
p&q&r
\end{pmatrix}.
\]
Substitution into the stabilizer equation gives $
p=q=0$, $r=1$, and $xy=1$. 
Therefore
\[
\operatorname{Stab}(A_{N_1})
=
\left\{
\operatorname{diag}(s,s^{-1},1):
s\in\mathbb R \backslash \{0\}
\right\}.
\]
For
$g=\operatorname{diag}(s,s^{-1},1)$
the transformation law
\[
P'(u)
=
gP\bigl((\det g)^{-1}g^Tu\bigr)g^{-1}
\]
gives
$t\longmapsto t/s$.
Hence all nonzero values of $t$ are equivalent to $t=1$, whereas
$t=0$ is a separate orbit.

Summarizing, we conclude:

\begin{theor}\label{thm:rank-two}
Let $B$ be a simple three-dimensional real Bol algebra with
$\dim[B,B]=2$.
Then, up to isomorphism, its binary product is represented by
\[
A=
\begin{pmatrix}
0&1&0\\
0&0&0\\
0&0&1
\end{pmatrix},
\]
and its ternary product is represented by
\[
P_1=0,
\qquad
P_2=
\begin{pmatrix}
-\dfrac{2t}{3}&0&0\\[3mm]
\dfrac{t^3}{36}&\dfrac t3&0\\[3mm]
-\dfrac{5t^2}{12}&1&\dfrac{2t}{3}
\end{pmatrix},
\qquad
P_3=
\begin{pmatrix}
2&0&0\\[1mm]
0&-2&-\dfrac t3\\[1mm]
t&0&-1
\end{pmatrix},
\]
where
$t\in\{0,1\}$.

Conversely, both of these Bol algebras are simple.
\end{theor}

\subsection{The rank-one case} \label{sec:3.3}

We begin by reducing the binary product to normal form.

\begin{lem}\label{lem:rank-one-normal-forms}
Let \(A\in M_3(\mathbb R)\) have rank one. Up to the twisted action \eqref{eq:action} 
the matrix \(A\) is equivalent to exactly one of the two matrices
\[
 A_{\mathrm s}=
 \begin{pmatrix}
 1&0&0\\
 0&0&0\\
 0&0&0
 \end{pmatrix},
 \qquad
 A_{\mathrm n}=
 \begin{pmatrix}
 0&1&0\\
 0&0&0\\
 0&0&0
 \end{pmatrix}.
\]
\end{lem}

\begin{proof}
Since \(\operatorname{rank}A=1\), there exist nonzero vectors
\(u,v\in\mathbb R^3\) such that
$A=uv^T$.
Under ordinary congruence one has
\[
kAk^T=(ku)(kv)^T.
\]
As observed in Lemma~\ref{lem:twisted-congruence}, the twisted action
has the same orbits as ordinary congruence modulo multiplication by
\(-1\).
There are now two cases.

If \(u\) and \(v\) are linearly dependent, say \(v=\lambda u\) with
\(\lambda\neq0\), then
$A=\lambda uu^T$
is symmetric. Choosing \(k\in\operatorname{GL}_3(\mathbb R)\) such that
\(ku=|\lambda|^{-1/2}e_1\), we obtain
\[
kAk^T=\operatorname{sgn}(\lambda)E_{11}.
\]
Modulo the overall sign, this gives \(A_{\mathrm s}\).

If \(u\) and \(v\) are linearly independent, choose
\(k\in\operatorname{GL}_3(\mathbb R)\) such that
$ku=e_1$, and $kv=e_2$,
thus
\[
kAk^T=e_1e_2^T=E_{12}=A_{\mathrm n}.
\]

Finally, the two resulting matrices cannot be equivalent, since
\(A_{\mathrm s}\) is symmetric whereas \(A_{\mathrm n}\) is not, and
symmetry is preserved by congruence and by multiplication by \(-1\).
\end{proof}

\subsubsection{The symmetric representative}

Assume \(A=A_{\mathrm s}\), put
$W=\operatorname{Span}\{e_2,e_3\}$,
and let $\pi:B\to W$ be the projection along $\mathbb Re_1$.
A direct expansion of (B1$_{\mathrm{op}}$) and (C$_{\mathrm{op}}$) shows that, with
respect to the basis \((e_2,e_3)\) of \(W\), the \(W\)-components of the
three operators have the form
\begin{equation} \label{eq3.3.1.1}
\pi P_1=
\begin{pmatrix}
\dfrac{f-b}{2}&-z&x\\[1mm]
a+d&y&z
\end{pmatrix},
\qquad
\pi P_2=
\begin{pmatrix}
0&-\dfrac{b+f}{2}&e\\[1mm]
0&-a-2d&f
\end{pmatrix},
\qquad
\pi P_3=
\begin{pmatrix}
0&a&b\\
0&c&d
\end{pmatrix}
\end{equation} 
for suitable real parameters \(a,b,c,d,e,f,x,y,z\). Only these
components will be needed below.

Among the equations \rm{(B2$_{ij}$)} one finds
\begin{equation} \label{eq3.3.1.2}
c(a+d)=0,\qquad
cf=d(2a+d),\qquad
ce=b(2a+d), 
\end{equation} 
\begin{equation} \label{eq3.3.1.3}
2ad+bc+cf+4d^2=0,
\qquad
4a^2+2ad+bc+cf=0,
\end{equation} 
\begin{equation} \label{eq3.3.1.4}
b^2+bf-2de=0,
\qquad
bf+2de-3f^2=0,
\end{equation}
and
\begin{equation} \label{eq3.3.1.5}
-3ab+af-4bd+2df=0.
\end{equation}
Subtracting the two equations in \eqref{eq3.3.1.3} gives
$a^2=d^2$.

If \(a+d\neq0\), then necessarily \(a=d\neq0\). Equation \eqref{eq3.3.1.2}
then gives \(c=0\), and its second equation gives \(0=3d^2\), a
contradiction. Hence
\begin{equation} \label{eq3.3.1.6}
a=-d.
\end{equation}

If \(c=0\), the second equation in \eqref{eq3.3.1.2} gives \(d=0\), and hence
\(a=0\). Equations \eqref{eq3.3.1.4} then imply \(b=f=0\). Thus in this
case
\begin{equation} \label{eq3.3.1.7}
a=b=c=d=f=0,
\qquad e\in\mathbb R. \nonumber 
\end{equation} 
If \(e=0\), then \(\pi P_2=\pi P_3=0\) and \(\pi P_i e_1=0\) for
\(i=1,2,3\). Since \([e_1,B]=0\) and
\(e_1\times B=W\), Lemma~\ref{bolideal} shows that
\(\mathbb Re_1\) is a proper Bol ideal.

Assume therefore \(e\neq0\). The remaining equations
\rm{(B2$_{ij}$)} force
\begin{equation} \label{eq3.3.1.8}
y=z=0.
\end{equation}
Indeed, if \(\lambda=(P_3)_{13}\) and \(\mu=(P_2)_{12}\), the relevant
entries give
\[
ey=0,\qquad e\lambda=0,
\qquad e(-\lambda+3z-\mu)=0,
\qquad e(2z+\mu)=0,
\]
and \eqref{eq3.3.1.8} follows. By \eqref{eq3.3.1.1}, all \(P_i(B)\) are then
contained in
\[
I=\operatorname{Span}\{e_1,e_2\}.
\]
Since \([B,B]=\mathbb Re_1\subset I\), Lemma~\ref{bolideal}
shows that \(I\) is a proper Bol ideal.

It remains to consider \(c\neq0\). From \eqref{eq3.3.1.2}, and \eqref{eq3.3.1.4}-\eqref{eq3.3.1.6} one obtains
\begin{equation} \label{eq3.3.1.9}
b=f=-\frac{d^2}{c},
\qquad
e=\frac{d^3}{c^2}.
\end{equation} 
Indeed, if \(d=0\), equations \eqref{eq3.3.1.2} and \eqref{eq3.3.1.4} give
\(b=e=f=0\); if \(d\neq0\), equation \eqref{eq3.3.1.5} first gives
\(b=f\), and the formulas follow from \eqref{eq3.3.1.2}.

The remaining equations \rm{(B2$_{ij}$)} give
\begin{equation} \label{eq3.3.1.10}
dy=cz,
\qquad
cx+dz=0.
\end{equation}
For completeness, if
\[
\lambda=(P_3)_{12},\qquad
\mu=(P_3)_{13},\qquad
\nu=(P_2)_{12},
\]
the required entries are
\[
c\mu-2cz-d\lambda+2dy=0,
\qquad
c\nu+d\lambda=0,
\]
\[
c\mu+3cz+c\nu-3dy=0,
\]
\[
c^2x-cd\mu-cdz-cd\nu+2d^2y=0.
\]
The first three equations yield \(dy=cz\) and \(\mu+\nu=0\), and the
last one then gives \(cx+dz=0\).

Put
\[
w=-d e_2+c e_3.
\]
Using \eqref{eq3.3.1.9} and \eqref{eq3.3.1.10} in \eqref{eq3.3.1.1}, we see that
\[
\pi P_i(B)\subseteq\mathbb Rw,
\qquad i=1,2,3.
\]
Consequently
$$
P_i(B)\subseteq I:=\operatorname{Span}\{e_1,w\},$$ for $i=1,2,3$.
Again $[B,B]=\mathbb Re_1\subset I$, so
Lemma \ref{bolideal} shows that $I$ is a proper Bol ideal.
We have proved the following:

\begin{prop}\label{propositionsymmetric-rank-one}
No Bol algebra with \(A=A_{\mathrm s}\) is simple.
\end{prop}

\subsubsection{The nonsymmetric representative}

Assume now \(A=A_{\mathrm n}\). Solving \rm{(B1$_{\mathrm{op}}$)} together with
cyclicity gives
$$
P_1=
\begin{pmatrix}
-\rho/2&0&\alpha\\
0&\beta&\gamma\\
0&0&0
\end{pmatrix},
\qquad
P_2=
\begin{pmatrix}
-\eta&-\rho/2&\delta\\
\varepsilon&-\kappa&\varphi\\
\zeta&0&\eta
\end{pmatrix},
\qquad
P_3=
\begin{pmatrix}
0&0&\rho\\
0&\lambda&\kappa\\
0&0&0
\end{pmatrix}.
$$
The second Bol identity now gives, among its entries,
\[
\lambda^2=0,
\qquad
\alpha\lambda-\beta^2=0,
\qquad
-2\alpha\lambda+3\rho^2=0.
\]
Hence
\begin{equation} \label{eq3.3.2.2}
\lambda=\beta=\rho=0.
\end{equation}
The remaining equations imply
\begin{equation} \label{eq3.3.2.3}
\alpha\varepsilon=
\alpha\zeta=
\alpha\eta=
\alpha\kappa=0,
\end{equation}
\begin{equation} \label{eq3.3.2.4}
\gamma\zeta=0,
\qquad
\gamma(\eta+2\kappa)=0,
\qquad
\kappa\delta-2\gamma\eta=0,
\end{equation}
\begin{equation} \label{eq3.3.2.5}
\gamma\varepsilon+\varphi\kappa=0,
\qquad
\zeta\kappa=0,
\qquad
\kappa(\eta+2\kappa)=0.
\end{equation}
These consequences are already sufficient to exclude simplicity.

If \(\alpha\neq0\), then \eqref{eq3.3.2.3} gives $
\varepsilon=\zeta=\eta=\kappa=0$. 
Together with \eqref{eq3.3.2.2}, this shows that 
$P_i(B)\subseteq I:=\operatorname{Span}\{e_1,e_2\}$, for $i=1,2,3$. 
Since \([B,B]=\mathbb Re_1\subset I\), the plane \(I\) is a proper Bol
ideal.

Assume next \(\alpha=0\) and \(\kappa\neq0\). 
In this case both \(P_1(B)\) and \(P_3(B)\) are contained in
\(\mathbb Re_2\). Moreover \([e_2,B]=0\) and
$
e_2\times B=\operatorname{Span}\{e_1,e_3\}$. 
Lemma~\ref{bolideal} therefore shows that
\(\mathbb Re_2\) is a proper Bol ideal.

Assume now \(\alpha=\kappa=0\) and \(\gamma\neq0\). Equations
\eqref{eq3.3.2.4}--\eqref{eq3.3.2.5} give
$
\eta=\zeta=\varepsilon=0$. 
Hence again
$ 
P_i(B)\subseteq\operatorname{Span}\{e_1,e_2\}$,
for $i=1,2,3$, 
so \(\operatorname{Span}\{e_1,e_2\}\) is a proper Bol ideal.

Finally, suppose \(\alpha=\kappa=\gamma=0\). Then
\(P_1=P_3=0\). 
Since
\([e_2,B]=0\) and
$
e_2\times B=\operatorname{Span}\{e_1,e_3\}$, 
Lemma~\ref{bolideal} shows once more that
\(\mathbb Re_2\) is a proper Bol ideal.

Combining the nonsymmetric case with
Proposition~\ref{propositionsymmetric-rank-one}, we obtain the rank-one
classification.

\begin{theor}\label{thm:rank-one}
No simple three-dimensional real Bol algebra has binary product of rank
one. Equivalently, if \(\operatorname{rank}A=1\), then \(B\) contains a
nonzero proper Bol ideal.
\end{theor}

\subsection{Simple three-dimensional Lie triple systems} \label{sec:3.4}

It remains to consider the case
$[B,B]=0$.
Then the binary product vanishes identically and the Bol identities reduce
precisely to the identities of a Lie triple system. We use the classical
structure theory of Lister~\cite{Lister}.

Recall that every Lie triple system $B$ admits a one-to-one universal
embedding in a Lie algebra $\mathfrak g_U(B)$, and
\[
\mathfrak g_U(B)
=
B\oplus [B,B]_{\mathfrak g_U(B)}.
\]
Here $[B,B]_{\mathfrak g_U(B)}$ denotes the Lie bracket in the universal
Lie algebra and should not be confused with the binary Bol product, which
is zero in the present case.

By Lister~\cite[Theorem~2.13]{Lister}, if $B$ is a simple Lie triple
system, then exactly one of the following alternatives occurs:

\begin{enumerate}
\item $B$ is the Lie triple system of a simple Lie algebra $\mathfrak h$,
with
$\langle x,y,z\rangle=[[x,y],z]$;
\item $B$ is not the Lie triple system of a Lie algebra and its universal
Lie algebra $\mathfrak g_U(B)$ is simple.
\end{enumerate}
We specialize this dichotomy to $\dim B=3$.

Suppose first that alternative 1. occurs. Then $\mathfrak h$ is a
three-dimensional simple real Lie algebra, and hence
$\mathfrak h\simeq\mathfrak{so}(3)$ or
$\mathfrak h\simeq\mathfrak{sl}_2(\mathbb R)$.

For $\mathfrak{so}(3)$, choose a basis satisfying
\[
[e_1,e_2]=e_3,\qquad
[e_2,e_3]=e_1,\qquad
[e_3,e_1]=e_2.
\]
Then, for $i\ne j$, one has
$[[e_i,e_j],e_j]=-e_i$,
whereas the triple brackets involving three distinct basis vectors vanish.
Hence
\begin{equation}\label{eq:LTSmetric}
\langle x,y,z\rangle
=
h(y,z)x-h(x,z)y
\end{equation}
with
$h=-I_3$.

For $\mathfrak{sl}_2(\mathbb R)$, we may choose a basis satisfying
\[
[e_1,e_2]=-e_3,\qquad
[e_2,e_3]=e_1,\qquad
[e_3,e_1]=e_2.
\]
Indeed,
\[
[[e_1,e_2],e_2]=e_1,\qquad
[[e_2,e_3],e_3]=-e_2,\qquad
[[e_3,e_1],e_1]=e_3,
\]
and again the triple brackets involving three distinct basis vectors vanish.
Thus \eqref{eq:LTSmetric} holds with
$h=\operatorname{diag}(1,1,-1)$.

Consider now alternative 2., and put
\[
\mathfrak g=\mathfrak g_U(B),
\qquad
\mathfrak h=[B,B]_{\mathfrak g}.
\]
Then
$\mathfrak g=B\oplus\mathfrak h$.
Since the bracket induces a surjective linear map
\[
\Lambda^2B\longrightarrow\mathfrak h,
\qquad
x\wedge y\longmapsto[x,y]_{\mathfrak g},
\]
we have
\[
\dim\mathfrak h\leq\dim\Lambda^2B=3,
\]
and consequently
$3\leq\dim\mathfrak g\leq6$.

By the classification of low-dimensional real Lie algebras
\cite{Turkowski}, the only simple real Lie algebras of dimension
less than eight are
\[
\mathfrak{so}(3),\qquad
\mathfrak{sl}_2(\mathbb R),\qquad
\mathfrak{so}(3,1)
\simeq
\mathfrak{sl}_2(\mathbb C)_{\mathbb R}.
\]
The first two have dimension three. If $\dim\mathfrak g=3$, then
$\mathfrak h=0$, and consequently
\[
\langle x,y,z\rangle
=
[[x,y]_{\mathfrak g},z]_{\mathfrak g}
=0,
\]
contrary to the simplicity of the nonzero Lie triple system $B$.
Therefore
\begin{equation}\label{eq:dimg}
\dim\mathfrak g=6,
\qquad
\dim\mathfrak h=3,
\end{equation}
and
$\mathfrak g
\simeq
\mathfrak{sl}_2(\mathbb C)_{\mathbb R}$.
By Lister~\cite[Theorem~1.1]{Lister}, the decomposition
$\mathfrak g=\mathfrak h\oplus B$
is the eigenspace decomposition of a unique involutive automorphism
$\sigma$ of $\mathfrak g$:
\[
\mathfrak h=\mathfrak g^\sigma,
\qquad
B=\mathfrak g^{-\sigma}.
\]

Let $J$ denote multiplication by $i$ on
$\mathfrak g=\mathfrak{sl}_2(\mathbb C)_{\mathbb R}$, and
consider the complexification
\[
\mathfrak g_{\mathbb C}
=
\mathfrak g\otimes_{\mathbb R}\mathbb C
\simeq
\mathfrak{sl}_2(\mathbb C)\oplus
\mathfrak{sl}_2(\mathbb C).
\]
The two simple ideals in this decomposition are precisely the
$+i$- and $-i$-eigenspaces of the complexification $J_{\mathbb C}$.

The automorphism $\sigma$ extends complex-linearly to an automorphism
$\sigma_{\mathbb C}$ of $\mathfrak g_{\mathbb C}$. Since every automorphism
of a semisimple Lie algebra permutes its simple ideals,
$\sigma_{\mathbb C}$ either preserves the two eigenspaces of
$J_{\mathbb C}$ or interchanges them. In the first case
$\sigma J=J\sigma$,
whereas in the second
$\sigma J=-J\sigma$.
Thus
$\sigma J\sigma^{-1}=\pm J$.

The positive sign cannot occur. Indeed, if $\sigma J=J\sigma$, then
$\mathfrak h$ is $J$-invariant, since for every $x\in\mathfrak h$,
\[
\sigma(Jx)=J\sigma(x)=Jx.
\]
Thus $\mathfrak h$ would be a complex vector space and would therefore
have even real dimension, contrary to \eqref{eq:dimg}. Hence
\begin{equation}\label{eq:antiJ}
\sigma J=-J\sigma.
\end{equation}

For $x\in\mathfrak h$, equation \eqref{eq:antiJ} gives
\[
\sigma(Jx)=-J\sigma(x)=-Jx,
\]
and therefore
$J\mathfrak h\subseteq B$.
Since both spaces have real dimension three,
$B=J\mathfrak h$.

Consequently
\[
\mathfrak g=\mathfrak h\oplus J\mathfrak h
\]
as real vector spaces. Regard
\(\mathfrak g=\mathfrak{sl}_2(\mathbb C)_{\mathbb R}\)
as the complex Lie algebra \(\mathfrak{sl}_2(\mathbb C)\), with
multiplication by \(i\) given by \(J\). Consider
\[
\Phi:\mathfrak h\otimes_{\mathbb R}\mathbb C
   \longrightarrow \mathfrak{sl}_2(\mathbb C),
\qquad
\Phi\bigl(x\otimes(a+ib)\bigr)=ax+bJx.
\]
The decomposition
\(\mathfrak g=\mathfrak h\oplus J\mathfrak h\)
shows that \(\Phi\) is an isomorphism of complex vector spaces.
Moreover, since \(\mathfrak h=\mathfrak g^\sigma\) is a real Lie
subalgebra and the bracket of \(\mathfrak{sl}_2(\mathbb C)\) is
complex bilinear,
\[
[Jx,y]=J[x,y],\qquad
[x,Jy]=J[x,y],\qquad
[Jx,Jy]=-[x,y]
\]
for all \(x,y\in\mathfrak h\).
It follows that \(\Phi\) preserves Lie brackets. Hence
\[
\mathfrak h\otimes_{\mathbb R}\mathbb C
   \simeq \mathfrak{sl}_2(\mathbb C)
\]
as complex Lie algebras, and therefore \(\mathfrak h\) is a real
form of \(\mathfrak{sl}_2(\mathbb C)\). In particular, $\mathfrak h$ is a
three-dimensional real simple Lie algebra: indeed, the complexification
of any nonzero ideal of $\mathfrak h$ would be a nonzero ideal of
$\mathfrak{sl}_2(\mathbb C)$. Hence, by the same low-dimensional
classification used above,
\[
\mathfrak h\simeq\mathfrak{so}(3)
\qquad\text{or}\qquad
\mathfrak h\simeq\mathfrak{sl}_2(\mathbb R).
\]

Identifying $B=J\mathfrak h$ with $\mathfrak h$ by
$Jx\longmapsto x$,
and using the complex bilinearity of the Lie bracket, we obtain
\[
[Jx,Jy]=-[x,y],
\]
and hence
\[
[[Jx,Jy],Jz]
=
-J[[x,y],z].
\]
Thus alternative 2. gives precisely the negatives of the two
Lie triple systems arising in alternative 1., and, in terms of
\eqref{eq:LTSmetric}, these correspond respectively to
$h=I_3$, and
$h=\operatorname{diag}(-1,-1,1)$.

We have therefore proved the classification part of the following theorem.

\begin{theor}
\label{thm:simpleLTS}
Let $B$ be a three-dimensional simple real Lie triple system. Then there
exists a nondegenerate symmetric bilinear form $h$ on $B$ such that
\begin{equation}\label{eq:Tg}
\langle x,y,z\rangle
=
h(y,z)x-h(x,z)y.
\end{equation}
Conversely, every product of the form \eqref{eq:Tg}, with $h$
nondegenerate, defines a simple Lie triple system.

Moreover, two such Lie triple systems $B_{h_1}$ and $B_{h_2}$ are isomorphic
if and only if the corresponding symmetric bilinear forms $h_1$ and $h_2$
are congruent. Consequently, up to isomorphism, there are exactly four
simple three-dimensional real Lie triple systems, represented by
\[
h=\pm I_3,\qquad 
h=\pm \operatorname{diag}(1,1,-1).
\]
\end{theor}

\begin{proof}
Only the converse and the assertion concerning isomorphisms remain to
be proved.

Let $h$ be nondegenerate and define the ternary product by
\eqref{eq:Tg}. A direct substitution shows that it satisfies the Lie
triple system identities.

We prove simplicity. Let $I\neq0$ be an ideal and choose
$0\neq x\in I$. For any $y\notin\mathbb Rx$, the linear functionals
$h(x,\cdot)$ and $h(y,\cdot)$
are linearly independent. Indeed, their linear dependence would, by
the nondegeneracy of $h$, imply the linear dependence of $x$ and $y$.
Hence there exists $z\in B$ such that
\[
h(x,z)=1,\qquad
h(y,z)=0.
\]
It follows that
\[
\langle x,y,z\rangle=-y\in I.
\]
Thus every $y\notin\mathbb Rx$ belongs to $I$, and therefore $I=B$.
Hence $B_h$ is simple.

Finally, suppose that
$F:B_{h_1}\longrightarrow B_{h_2}$
is an isomorphism. Then
\[
F\bigl(\langle x,y,z\rangle_{h_1}\bigr)
=
\langle Fx,Fy,Fz\rangle_{h_2},
\]
so
\[
h_1(y,z)Fx-h_1(x,z)Fy
=
h_2(Fy,Fz)Fx-h_2(Fx,Fz)Fy.
\]
Whenever $x$ and $y$ are linearly independent, so are $Fx$ and $Fy$,
and comparison of coefficients gives
\[
h_2(Fy,Fz)=h_1(y,z),
\qquad
h_2(Fx,Fz)=h_1(x,z).
\]
Let $0\neq u\in B$ be arbitrary. Choose $x\notin\mathbb Ru$ and take
$y=u$ in the preceding equality. Since $z$ is arbitrary, we obtain
\[
h_2(Fu,Fz)=h_1(u,z),
\]
for $z\in B$, hence
\[
h_2(Fu,Fv)=h_1(u,v),
\]
for $u,v\in B$,
the case $u=0$ being trivial. Thus $F$ is an isometry from $h_1$ to $h_2$.

Conversely, every isometry from $h_1$ to $h_2$ is plainly an isomorphism
from $B_{h_1}$ to $B_{h_2}$. Therefore
\[
B_{h_1}\simeq B_{h_2} 
\quad\Longleftrightarrow\quad
h_1\text{ and }h_2\text{ are congruent}.
\]
By Sylvester's law of inertia, the four matrices displayed in the
statement represent exactly the four congruence classes of
nondegenerate symmetric bilinear forms on a three-dimensional real
vector space.
\end{proof}

\bigskip
Giovanni Falcone, Department of Mathematics and Computer Science, University of Palermo, Italy, Email: giovanni.falcone@unipa.it, 
ORCID: 0000-0002-5210-5416 

\bigskip
\'Agota Figula$^{\ast}$, Institute of Mathematics, University of Debrecen, P.O.B. 400, H-4002, Deb\-recen, Hungary 
Email: figula@science.unideb.hu, 
ORCID: 0000-0002-8095-6074

\bigskip
Emese K\'asa, Doctoral School of Mathematical and Computational Sciences, University of Deb\-recen,  P.O.B. 400, H-4002, Debrecen, Hungary, 
Email:  kasa.emese@science.unideb.hu, 
ORCID: 0009-0002-0804-545X \\
and \\
Institute of Mathematics and Computer Sciences, University of Ny\'iregyh\'aza, P.O.B. 166, H-4400, Ny\'iregyh\'aza, Hungary 

\bigskip
Gianmarco Mattana, Department of Mathematics and Computer Science, University of
Palermo, Italy, Email: gianmarco.mattana@unipa.it 

\bigskip 
P\'eter T. Nagy, John von Neumann Faculty of Informatics, \'Obuda University,  H-1034, Budapest, B\'ecsi \'ut 96/b, Hungary, 
Email: nagy.peter@nik.uni-obuda.hu,
ORCID: 0000-0001-6838-504

\end{document}